\documentclass[reqno]{amsart}
\usepackage{amssymb}
\usepackage{amsmath}
\usepackage{amsthm}
\usepackage{mathtools}
\usepackage{graphicx}
\usepackage{comment}
\usepackage{caption}
\usepackage{calc}
\usepackage[font=small,labelfont=small]{caption}
\usepackage[style=numeric,maxbibnames=99]{biblatex}
\usepackage{float}
\usepackage{csquotes}
\DeclareNameAlias{default}{last-first}

\usepackage{tikz}
\usepackage{tikz-3dplot}
\tdplotsetmaincoords{70}{110} 

\usetikzlibrary{calc}
\usetikzlibrary{arrows}

\usepackage{comment}

\usepackage{hyperref}
\hypersetup{colorlinks}

\usepackage{amssymb, mathrsfs, amsfonts, amsmath}

\makeatletter
\def\part#1{%
  \refstepcounter{part}%
  \bigskip
  \begin{center}
    \normalfont\bfseries
    \Large Part~\thepart: #1
  \end{center}
  \medskip
}
\makeatother

\usepackage{tikz}
\usetikzlibrary{arrows}

\usepackage[english]{babel}

\usepackage[margin=1.4in, centering]{geometry}

\DeclareSymbolFont{bbold}{U}{bbold}{m}{n}
\DeclareSymbolFontAlphabet{\mathbbold}{bbold}
\newtheorem{thm}{Theorem}[section]
\newtheorem{prop}[thm]{Proposition}
\newtheorem{lem}[thm]{Lemma}
\newtheorem{cor}[thm]{Corollary}
\theoremstyle{definition}
\newtheorem{defn}[thm]{Definition}

\theoremstyle{remark}
\newtheorem{rem}[thm]{Remark}

\newtheorem{conj}[thm]{Conjecture}

\usepackage{graphicx} 

\newcommand{\R}{\mathbb{R}}

\newcommand{\E}{\mathcal{E}}
\newcommand{\A}{\textnormal{Area}}

\tikzset{vtx/.style={circle, fill, inner sep=1.5pt}}

\title[On volume vectors determined by hypergraphs]{On volume vectors determined by hypergraphs in thin subsets of Euclidean space}

\author[T. Borges, B. Foster, Y. Ou, E. Palsson, F. Romero Acosta]{Tainara Borges, Ben Foster, Yumeng Ou, Eyvindur Palsson, Francisco Romero Acosta}

\address[T. Borges]{Department of Mathematics, University of Pennsylvania, Philadelphia, PA 19104, USA}\email{tborges@sas.upenn.edu}
\address[B. Foster]{Department of Mathematics, Washington University in St. Louis, St. Louis, MO 63130, USA}\email{bfoster@wustl.edu}
\address[Y. Ou]{Department of Mathematics, University of Pennsylvania, Philadelphia, PA 19104, USA}\email{yumengou@sas.upenn.edu}
\address[E. Palsson]{Department of Mathematics, Virginia Tech, Virginia, VA 24061, USA}\email{palsson@vt.edu}
\address[F. Romero Acosta]{MSI, The Australian National University, ACT, Australia}\email{Francisco.Romero@anu.edu.au}

\begin{document}
\maketitle

\begin{abstract}
Generalizing the Falconer distance problem, the authors of this paper recently established the first non-trivial dimensional threshold for any distance graph in high enough of a dimension. The methods developed were flexible enough to generalize from the Euclidean distance to any two point configuration, conditional on results on $k$-stars for the two point configuration. A natural question emerges on what happens to configurations that take in more than two points. In this paper we consider a classic three point variant of the Falconer distance problem, namely that on areas of triangles and its generalizations to volumes of simplices. In this model case we develop two methods. One we call the Jacobian method which allows us, through Heron's formula, to leverage earlier results on distance graphs and obtains non-trivial thresholds for volume vectors determined by a wide range of hypergraphs of simplices. Even in the classic case of the volume of a single simplex this method yields the best known dimensional thresholds if the dimension is considerably bigger than the size of the simplex. We develop a conjecture that has connections to rigidity theory. The Jacobian method works best in high dimensions so in the case of areas of triangles in the plane, we refine the work of Shmerkin and Yavicoli, who recently resolved a conjecture for areas of triangles in the plane, and obtain building blocks from which we can get abundance of area vectors determined by certain hypergraphs of triangles, such as chains of triangles connected on edges or vertices. The results improve and extend existing results of Galo and McDonald as well as of Greenleaf, Iosevich and Taylor.
\end{abstract}

\section{Introduction}

The Falconer distance problem, which can be viewed as a continuous analog of the Erd\H{o}s distinct distance problem, asks for a compact set $E\subset\mathbb{R}^d$, $d\geq 2$, how large its Hausdorff dimension $\dim_{H}(E)$ needs to be to guarantee that its distance set
$$ \Delta(E) = \lbrace |x-y| : x,y\in E \rbrace $$
has positive Lebesgue measure. Falconer showed the threshold $\frac{d}{2}$ was necessary and conjectured it was sufficient \cite{Falconer85}. In recent years there has been much progress on this conjecture with the current best thresholds being $\frac{5}{4}$ when $d=2$ and $\frac{d}{2} + \frac{1}{4} - \frac{1}{8d+4}$ when $d\geq 3$ \cite{GIOW20,DORZ23}. A classic variant, the pinned Falconer distance problem, similarly asks how large $\dim_{H}(E)$ needs to be to guarantee the existence of a pin $x\in E$ such that the pinned distance set
$$ \Delta_{x}(E) = \lbrace |x-y| : y\in E \rbrace $$
has positive Lebesgue measure. Any threshold established for the pinned problem implies the same threshold works for the original problem. Surprisingly, due to a conversion mechanism by Liu \cite{Liu19}, all the best thresholds for the Falconer distance problem also hold for the pinned variant.

A big push in the field has been towards more complicated configurations built out of distances. In the greatest generality one can consider a distance graph. More precisely, for a graph $G=(\mathcal{V},\mathcal{E})$ denote the $G$-distance set of a compact subset $E\subset\mathbb{R}^d$ by
$$ \Delta^{G}(E):= \lbrace (|x_{i}-x_{j}|)_{(v_i,v_j)\in\mathcal{E}} : x_1,\ldots,x_{|\mathcal{V}|}\in E \rbrace $$
where $\mathcal{V}=\{v_1,v_2,\dots ,v_{|\mathcal{V}|}\}$ is the set of vertices of $G$ and, for convenience, we represent the edges in $\mathcal{E}$ as ordered pairs $(v_i,v_j)$ with $i<j$. We say the graph $G$ is \emph{non-trivial} in this setting if $|\mathcal{E}|\geq 1$ so the $G$-distance set is not vacuous.  To take an example of a $G$-distance set, if $G=K_3$, the complete graph on $3$ vertices, we obtain the set of triangles, namely $\Delta^{K_3}(E)=\lbrace (|x_1 - x_2|,|x_1 - x_3|,|x_2 - x_3|) : x_1,x_2,x_3 \in E\rbrace$. Numerous papers have appeared that address particular configurations such as chains and trees \cite{BIT16,IT19,OT22,pinnedtrees}, triangles and simplices \cite{GI12,EHI13,GGIP15,GILP15,GIT22,PRA23,GIT24,PRA25,IPPS22} and cycles \cite{GIP17,IMMM25}. Some of these papers develop frameworks within which results for individual configurations can be obtained, but an easily computable result for any graph $G$ remained elusive for a while.

In \cite{BFOPRpaperI} the authors of this paper obtained a non-trivial threshold for a Falconer type result for any non-trivial distance graph in a high enough dimension. A classic notion from graph theory is that a graph $G$ is said to be $k$-degenerate if $k$ is the least number such that every induced subgraph of $G$ contains a vertex with $k$ or fewer neighbors. Using this notion, the authors established that if $G=(\mathcal{V},\mathcal{E})$ is a non-trivial graph that is $k$-degenerate then for any compact set $E\subset\mathbb{R}^d$, $d>k$, with $\dim(E)>\frac{d+k}{2}$ the $|\mathcal{E}|$-dimensional Lebesgue measure of $\Delta^{G}(E)$ is positive. Further extensions were given to distance graphs with multiple pins. Not only did this result apply in great generality, but for individual configurations, such as cycles, this result further gave the best thresholds known.

Another popular generalization is to consider more general two point configurations $\Phi:\mathbb{R}^{d}\times\mathbb{R}^{d}\rightarrow\mathbb{R}^k$ and their corresponding configuration sets $ \Delta^{\Phi}(E) = \lbrace \Phi(x,y) : x,y\in E \} $. One example is the dot product, namely $\Phi(x,y)=x\cdot y$, as studied in \cite{EIT11,EHI13,BMS24}, but see also \cite{IMT12,GIT21}. In \cite{BFOPRpaperI} the authors of this paper obtained general results on $\Phi$-graphs, analogous to the results on distance graphs, but conditional on there existing results for certain building blocks called $k$-stars. For distances such results come from \cite{IPPS22} but are not known for other configurations.

The study of general configuration functions $\Phi$ that take in more than two points was initiated in \cite{GGIP15} and see further progress in \cite{GIT22,GIT24}. While these individually can encode complicated configurations, such as any distance graph, then an interesting question is what can be built from a single fixed configuration function $\Phi$. When $\Phi$ takes in more than two points, it does not make sense anymore to build more complicated configurations through graphs. Instead, one needs hypergraphs where edges connect multiple vertices, which describe the inputs for the function $\Phi$. A first step in this direction was achieved in \cite{GIT25}, where under strong non-empty interior type results on the base $\Phi$ configuration, the authors were able to build trees of such configurations. In this paper one of our main goals is to see whether ideas from \cite{BFOPRpaperI} carry over to this hypergraph setting.

We illustrate how our framework can be extended to hypergraphs by studying a particular example, namely areas of triangles. Let $\text{Area}(x_1,x_2,x_3)$ denote the area of the triangle with vertices $x_1, x_2, x_3$ and let
$$ A_{\Delta}(E) := \lbrace \text{Area}(x_1,x_2,x_3) : x_1,x_2,x_3\in E \rbrace $$
denote the set of areas of triangles realized in a compact set $E\subset\mathbb{R}^d$, $d\geq 2$. The area set can then be naturally pinned, which we denote by $ (A_{\Delta})_{x_1}(E) := \lbrace \text{Area}(x_1,x_2,x_3) : x_2,x_3\in E \rbrace $ for a single pin and for two pins we denote it by $ (A_{\Delta})_{x_1,x_2}(E) := \lbrace \text{Area}(x_1,x_2,x_3) : x_3\in E \rbrace $.

Areas of triangles have been studied extensively, with the first result in this direction having been established by Erdo\u{g}an, Hart and Iosevich \cite{EHI13}, who obtained the threshold $\frac{3}{2}$ for areas of triangles in $\mathbb{R}^2$ pinned at the origin which turns the problem into a statement about dot products. A case of particular interest is considering triangles in the plane, their natural ambient space. Grafakos, Greenleaf, Iosevich and the fourth listed author of this paper \cite{GGIP15}, obtained the dimensional threshold $\frac{5}{4}$ for getting a positive Lebesgue measure of $A_{\Delta}(E)$ and conjectured that the threshold $1$ should be sufficient. It is clear that the threshold $1$ is necessary by simply considering the example of $E$ being a line. Very recently Shmerkin and Yavicoli \cite{SY25} resolved this conjecture in a strong way where they obtained that there exist two pins $x_1,x_2\in E$ such that $(A_{\Delta})_{x_1,x_2}(E)$ has positive Lebesgue measure. Many other variants have also been established. Most closely related is the work of Greenleaf, Iosevich and Taylor, who obtained non-empty interior for areas of triangles in the plane at a threshold $\frac{5}{3}$ \cite{GIT22}. McDonald obtained results on configurations called area types where multiple areas of triangles were computed in the plane \cite{McDonald21} and this was extended to volumes in the matching ambient dimension by Galo and McDonald \cite{GM22}. In particular, positive appropriate Lebesgue measure of certain chains of areas of triangles pinned at the origin in the plane follow from \cite{GM22} at a dimensional threshold $\frac{3}{2}$. Further results on tuples of areas were obtained by Greenleaf, Iosevich and Taylor \cite{GIT24} with their most recent work \cite{GIT25} obtaining results on non-empty interior of trees of areas of triangles in the plane at the threshold $\frac{5}{3}$. Many of these results extend more generally to simplices and we show some generalizations in that direction too.

A natural starting point towards more complicated configurations built from areas of triangles could be the resolution of Shmerkin and Yavicoli of the threshold for obtaining a positive Lebesgue measure of areas of triangles in the plane. Their result has two drawbacks for our purposes. First, while the result of Shmerkin and Yavicoli establishes a double pinned result for areas of triangles in the plane, it does not quantify the number of good pins. Quantifying the number of pins is essential in our arguments that use triangles as building blocks for more complicated graphs. Second, for more complicated configurations, it may be necessary to go into a high enough dimension and there the result of Shmerkin and Yavicoli becomes far from sharp\footnote{The statement of Theorem 1.1 in \cite{SY25} claims the dimensional threshold $1$ for the double pinned area of triangles in any dimension $d\geq 2$. However, as confirmed through personal communication with the authors, the proof does not yield that when $d\geq 3$. The problem is that the proof finds lots of distances to a certain hyperplane in which the base of the triangle lies, but when $d\geq 3$ there could possibly be many distances to the hyperplane, while the distances to the base could be very few.}. By a slicing method, that iterates on their result, one can obtain a dimensional threshold of $d-1$. In $\mathbb{R}^d$, $d\geq 3$, we develop a method that we call the Jacobian method, that fundamentally builds on the idea that by Heron's formula the area of a triangle can be determined by its side lengths. Leveraging our earlier results in \cite{BFOPRpaperI} on distance graphs, we are able to get Falconer type results for a broad range of complicated configurations where we calculate areas of multiple triangles. We emphasize that even in the case of a pinned area of a single triangle our dimensional threshold of $\frac{d+2}{2}$ (with slight improvements when $d=3$ with a threshold $2$ being possible through a slicing argument) is new. In $\mathbb{R}^2$ we refine the result of Shmerkin and Yavicoli by quantifying the number of pins for a pinned area of triangles result in Theorem \ref{thm:trianglebuildingblockintheplane}. This allows us, through inductive graph building arguments, to obtain positive Lebesgue measure worth of pinned chains of areas of triangles connected on vertices in the plane at a dimensional threshold $1$, extending the results of Galo and McDonald, as well as results on chains of areas of triangles connected by edges in the plane. The setting of triangles connected by edges is a new phenomenon that appears in this hypergraph setting. Classically, only triangles attached at vertices had been studied in the literature, with a notable exception being the triangle chains of Galo and McDonald \cite{GM22} and a few configurations studied by Greenleaf, Iosevich and Taylor \cite{GIT22,GIT24}.

\subsection{Jacobian method}
In Section \ref{sec: jacobianmethod}, we develop what we call the Jacobian method. While a more classical approach would just take in points from a hypergraph and calculate areas of triangles designated by the hypergraph, we instead approach this based on the observation that, by Heron's formula, the area of a triangle can be computed in terms of its side lengths. As we are now detecting the presence of not just edges in a graph, but also $3$-cycle, we thus introduce the notion of a simplicial complex in Section \ref{sec: preliminaries}, in order to not just record the triangles that we compute the areas of, but also their side lengths. For a simplicial complex $G$, we denote by $(A_{\Delta})^G(E)$ the triangle area set determined by $G$, which for whichever realization of the graph in $E$, collects the vector with all areas of triangles that we include in the simplicial structure; see \eqref{eq area vector set} for the precise definition. Informally, we say that edge area transferal holds for $G$ if for a particular dimensional threshold, a Falconer type result for the distance graph that the edges of the triangles form can be transferred to $(A_{\Delta})^G(E)$, see Definition \ref{def: edgetransferal} for more details. 

Leveraging our previous results on Falconer type results for distance graphs in \cite{BFOPRpaperI}, our main result is the following.

\begin{thm}\label{thm: mainJacobianthm}
    Let $G$ be a $2$-dimensional simplicial complex with the property that one edge may be selected injectively from each triangle in such a way that no triangle has all of its edges selected. Then edge area transferal holds for $G$.
\end{thm}

Using the theorem above, we can, for example, find Hausdorff dimension thresholds on $E\subset \R^d$ to guarantee positive measure of vectors of triangle areas realized in a stack of four triangles like in Figure \ref{fig:stack4triangles}, with vertices varying in $E$.

\begin{figure}[h!]
    \centering
\begin{tikzpicture}
  \coordinate (v1) at (0,0);
  \coordinate (v2) at (2,0);
  \coordinate (v3) at (4,0);
  
  \coordinate (v4) at (1, {sqrt(3)});
  \coordinate (v5) at (3, {sqrt(3)});
  
  \coordinate (v6) at (2, {2*sqrt(3)});

  \draw[thick, blue] (v1) -- node[below, black] {$e_1$} (v2);
  \draw[thick, blue] (v2) -- node[below, black] {$e_3$} (v3);
  \draw[thick] (v3) -- node[below right] {} (v5);
  \draw[thick, blue] (v5) -- node[above right] {$e_4$} (v6);
  \draw[thick] (v6) -- node[above left] {} (v4);
  \draw[thick] (v4) -- node[below left] {} (v1);
  
  \draw[thick] (v2) -- node[left] {} (v4);
  \draw[thick,blue] (v4) -- node[below] {$e_2$} (v5);
  \draw[thick] (v5) -- node[right] {} (v2);

  \node at (2, {4*sqrt(3)/3}) {$T_4$};
  \node at (1, {sqrt(3)/3}) {$T_1$};
  \node at (2, 0.85) {$T_2$};
  \node at (3, {sqrt(3)/3}) {$T_3$};

  \filldraw (v1) circle (2.5pt) node[below left] {$v_1$};
  \filldraw (v2) circle (2.5pt) node[below=4pt] {$v_2$};
  \filldraw (v3) circle (2.5pt) node[below right] {$v_3$};
  \filldraw (v4) circle (2.5pt) node[left=3pt] {$v_4$};
  \filldraw (v5) circle (2.5pt) node[right=3pt] {$v_5$};
  \filldraw (v6) circle (2.5pt) node[above=4pt] {$v_6$};
\end{tikzpicture}
\caption{A $2$-dimensional simplicial complex that looks like a stack of $4$ triangles (a few edge labels were omitted), and an edge selection function $T_i\mapsto e_i$ that respects the hypothesis of Theorem \ref{thm: mainJacobianthm}. The blue edges are the ones in the image of the edge selection function.} 
\label{fig:stack4triangles}
\end{figure}
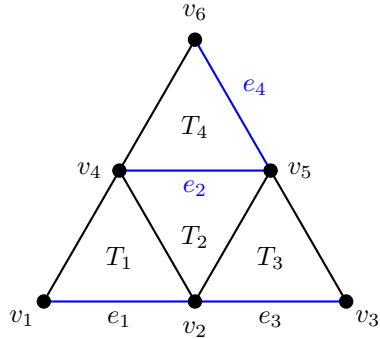

Figure \ref{fig:stack4triangles} illustrates how we can select edges $e_i$ for each triangle $T_i$ in a way such that, for each triangle, not all its edges are in the image of the edge selection function. Since this graph $(G,\mathcal{V},\mathcal{E})$ is $2$-degenerate, the distance set $$\Delta^{G}(E)=\{(|x_i-x_j|)_{(v_i,v_j) \in \mathcal{E}}\colon x_i\in E,\,1\leq i\leq 6\}\subset \R^9$$ 
has positive Lebesgue measure as long as $\dim(E)>\frac{d+2}{2}$ (as proved in \cite{BFOPRpaperI}) and the theorem above allows us to transfer that to positive $4$-dimensional Lebesgue measure of the area set
$$(A_{\Delta})^{G}(E)=\left\{(\A(x_1,x_2,x_4),\A(x_2,x_4,x_5),\A(x_2,x_3,x_5),\A(x_4,x_5,x_6))\colon \{x_i\}_{i=1}^{6}\subset E\right\}.$$

One can also apply the theorem above to get a positive $4$ dimensional Lebesgue measure of the vectors of triangle areas appearing in the faces of tetrahedron with vertices in $E\subset \R^d$, $d\geq 4$, and $\dim(E)>\frac{d+3}{2}$ (note that a tetrahedron is a $3$-degenerate graph, so we are again using the threshold provided by \cite{BFOPRpaperI} for positive measure of $\Delta^G(E)$). See Figure \ref{fig:K4simplicial} for an example of a suitable edge choice function.

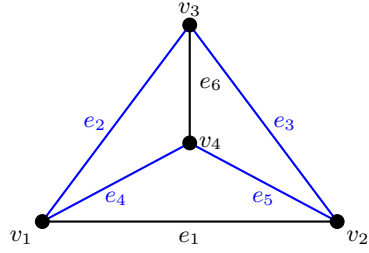
\begin{figure}[h!]
\centering
\begin{tikzpicture}[scale=1.3, every node/.style={font=\small}]
    \tikzset{vtx/.style={circle, fill=black, inner sep=1.8pt}}

    \coordinate (v1) at (-1.5,0);
    \coordinate (v2) at (1.5,0);
    \coordinate (v3) at (0,2);
    \coordinate (v4) at (0,0.8);

    \draw[thick] (v1) -- node[below] {$e_1$} (v2);
    \draw[thick,blue] (v1) -- node[left] {$e_2$} (v3);
    \draw[thick,blue] (v2) -- node[right] {$e_3$} (v3);
    \draw[thick,blue] (v1) -- node[below] {$e_4$} (v4);
    \draw[thick,blue] (v2) -- node[below] {$e_5$} (v4);
    \draw[thick] (v3) -- node[right] {$e_6$} (v4);

    \filldraw (v1) circle (2pt) node[below left] {$v_1$};
    \filldraw (v2) circle (2pt) node[below right] {$v_2$};
    \filldraw (v3) circle (2pt) node[above] {$v_3$};
    \filldraw (v4) circle (2pt) node[right] {$v_4$};


    \node[align=left, anchor=west] at (3.2,1.2) {
        $T_1:=(v_1,v_2,v_3)\mapsto e_2=(v_1,v_3)$\\
        $T_2:=(v_1,v_2,v_4)\mapsto  e_5=(v_2,v_4)$\\
        $T_3:=(v_1,v_3,v_4)\mapsto e_4=(v_1,v_4)$\\
        $T_4:=(v_2,v_3,v_4)\mapsto e_3=(v_2,v_3)$
    };
\end{tikzpicture}
\caption{The complete graph in $4$ vertices $K_4$ with a $2$-simplicial structure where $\mathcal{E}_2=\{T_1,T_2,T_3,T_4\}$ and an edge choice function that respects the hypothesis of Theorem \ref{thm: mainJacobianthm}.}
\label{fig:K4simplicial}
\end{figure}

We also prove a result that holds for pinned complexes. It requires an edge choice function that picks from each triangle an edge that is exclusive to that triangle, and that avoids picking pinned edges, that is, we cannot select an edge if both its comprising vertices are in the pinned vertex set. In the unpinned case, the hypothesis of the proposition below reduces to being able to select from each triangle an edge that is exclusive to that triangle, which is more restrictive than the hypothesis of Theorem \ref{thm: mainJacobianthm}.

\begin{prop}\label{prop: pinned suff transfer}
    Let $G=(\mathcal{V},\mathcal{E},\mathcal{P})$ be a pinned $2$-simplicial complex with pin set $\mathcal{P}\subset\mathcal{V}$ which has the property that for every $T\in \mathcal{E}_2$, there exists an edge $\{v,w\}\subset T$ with the property that $\{v,w\}\not\subset \mathcal{P}$ and for all $2$-simplices $T'\ne T$, we have that $\{v,w\}\not\subset T'$. Then length-area transferal holds for $G$.
\end{prop}

This proposition allows us to handle an array of pinned $2$-simplicial complexes, including wheels, pinned fishes, or circumscribed flowers like the ones in the table in Figure \ref{fig:pinned-simplicial-showcase}. The pinned fan $F_{1,5}$ in that table, as well as more general fans with all but one vertex pinned, are handled in Theorem \ref{thm: k pinned fans}. For more details and proofs for those configurations, see Section \ref{sec: jacobianmethod}.

\begin{figure}[h!]
\centering
\begingroup
\setlength{\tabcolsep}{7pt}
\setlength{\arrayrulewidth}{0.4pt}
\renewcommand{\arraystretch}{1.25}

\begin{tabular}{|c|c|}
\hline

\begin{minipage}[t]{0.43\linewidth}
\centering
\vspace{0.4em}

\begin{tikzpicture}[scale=0.82, >=latex, every node/.style={font=\scriptsize}]
\tikzset{
  showvtx/.style = {circle, inner sep=0pt, minimum size=5pt},
  showedge/.style = {line width=0.75pt}
}

\coordinate (w)  at (0,0);

\coordinate (v1) at (1,1.6);
\coordinate (v2) at (2,1.1);
\coordinate (v3) at (2.4,0);
\coordinate (v4) at (2,-1.1);
\coordinate (v5) at (1,-1.6);

\draw[showedge] (w) -- (v1);
\draw[showedge] (w) -- (v2);
\draw[showedge] (w) -- (v3);
\draw[showedge] (w) -- (v4);
\draw[showedge] (w) -- (v5);

\draw[showedge,magenta] (v1) -- (v2);
\draw[showedge,magenta] (v2) -- (v3);
\draw[showedge,magenta] (v3) -- (v4);
\draw[showedge,magenta] (v4) -- (v5);

\node[showvtx, fill=black, label=left:$w$] at (w) {};

\node[showvtx, fill=magenta, label=above:$v_1$] at (v1) {};
\node[showvtx, fill=magenta, label=right:$v_2$] at (v2) {};
\node[showvtx, fill=magenta, label=right:$v_3$] at (v3) {};
\node[showvtx, fill=magenta, label=right:$v_4$] at (v4) {};
\node[showvtx, fill=magenta, label=below:$v_5$] at (v5) {};
\end{tikzpicture}

\medskip
\textbf{Pinned fan \(F_{1,5}\)}

\[
\dim(E)>\frac{d+5}{2}
\]

\vspace{0.4em}
\end{minipage}
&
\begin{minipage}[t]{0.43\linewidth}
\centering
\vspace{0.4em}

\begin{tikzpicture}[scale=1.08, every node/.style={font=\scriptsize}]
\tikzset{
  showvtx/.style = {circle, fill=black, inner sep=1.7pt},
  showpin/.style = {circle, draw=magenta!70!black, fill=magenta, inner sep=1.9pt},
  showedge/.style = {line width=0.75pt}
}

\coordinate (w) at (0,0);

\foreach \i in {1,2,...,6} {
  \coordinate (v\i) at (360*\i/6+120:1);
  \draw[showedge] (w) -- (v\i);
}

\draw[showedge] (v1) -- (v2) -- (v3) -- (v4) -- (v5) -- (v6) -- (v1);

\draw[showedge,magenta] (w)--(v1);

\node[showpin, label=below:$y$] at (w) {};
\node[showpin] at (v1) {};
\foreach \i in {2,...,6} {
  \node[showvtx] at (v\i) {};
}

\node[left] at (v1) {$x_k$};
\node[left] at (v2) {$x_1$};
\node[left] at (v6) {$x_{k-1}$};
\end{tikzpicture}

\medskip
\textbf{Pinned wheel \(W_6\)}

\[
\dim(E)>\frac{d+3}{2}
\]

\vspace{0.4em}
\end{minipage}

\\
\hline

\begin{minipage}[t]{0.43\linewidth}
\centering
\vspace{0.4em}

\begin{tabular}{@{}c@{\hspace{0.15cm}}c@{}}
\begin{tikzpicture}[scale=0.42, every node/.style={font=\scriptsize}]
\tikzset{
  fishv/.style = {circle, fill=black, inner sep=2.1pt},
  fishpin/.style = {circle, draw=magenta!70!black, fill=magenta, inner sep=2.1pt},
  showedge/.style = {line width=0.8pt}
}

\coordinate (x3) at (0,0);
\coordinate (x1) at (2,1.2);
\coordinate (x2) at (2,-1.2);
\coordinate (x6) at (-2,1.2);
\coordinate (x4) at (-2,-1.2);
\coordinate (x5) at (-3.4,0);

\draw[showedge] (x1)--(x2)--(x3)--(x1);
\draw[showedge] (x6)--(x4)--(x5);
\draw[showedge,magenta] (x5)--(x6);
\draw[showedge] (x6)--(x3)--(x4);

\node[fishv, label=right:$v_4$] at (x3) {};
\node[fishv, label=above:$v_5$] at (x1) {};
\node[fishv, label=below:$v_6$] at (x2) {};
\node[fishpin,label=above:$v_2$] at (x6) {};
\node[fishv, label=below:$v_3$] at (x4) {};
\node[fishpin,label=left:$v_1$] at (x5) {};
\end{tikzpicture}
&
\begin{tikzpicture}[scale=0.42, every node/.style={font=\scriptsize}]
\tikzset{
  fishv/.style = {circle, fill=black, inner sep=2.1pt},
  fishpin/.style = {circle, draw=magenta!70!black, fill=magenta, inner sep=2.1pt},
  showedge/.style = {line width=0.8pt}
}

\coordinate (x3) at (0,0);
\coordinate (x1) at (2,1.2);
\coordinate (x2) at (2,-1.2);
\coordinate (x6) at (-2,1.2);
\coordinate (x4) at (-2,-1.2);
\coordinate (x5) at (-3.4,0);

\draw[showedge] (x1)--(x2)--(x3)--(x1);
\draw[showedge,magenta] (x2)--(x1);

\draw[showedge] (x6)--(x4)--(x5);
\draw[showedge,magenta] (x5)--(x6);

\draw[showedge] (x6)--(x3)--(x4);

\node[fishv, label=right:$v_4$] at (x3) {};
\node[fishpin, label=above:$v_5$] at (x1) {};
\node[fishpin, label=below:$v_6$] at (x2) {};
\node[fishpin,label=above:$v_2$] at (x6) {};
\node[fishv, label=below:$v_3$] at (x4) {};
\node[fishpin,label=left:$v_1$] at (x5) {};
\end{tikzpicture}
\end{tabular}

\medskip
\textbf{Pinned fish graph}

\[
\mathcal P=\{v_1,v_2\}:
\quad
\dim(E)>\frac{d+2}{2}
\]

\[
\mathcal P=\{v_1,v_2,v_5,v_6\}:
\quad
\dim(E)>\frac{d+3}{2}
\]

\vspace{0.4em}
\end{minipage}
&
\begin{minipage}[t]{0.43\linewidth}
\centering
\vspace{0.4em}

\begin{tikzpicture}[scale=0.72, every node/.style={font=\scriptsize}]
\foreach \i in {1,2,3,4} {
    \coordinate (v\i) at (90*\i+45:1);
    \coordinate (w\i) at (90*\i+90:2.3); 
}

\draw[line width=0.75pt] (w1) -- (v2) -- (w2) -- (v3) -- (w3) -- (v4) -- (w4) -- (v1);
\draw[line width=0.75pt] (v1) -- (v2) -- (v3) -- (v4) -- (v1);
\draw[line width=0.75pt] (w1) -- (w2) -- (w3) -- (w4) -- (w1);

\draw[magenta,line width=0.75] (v1)--(w1);
\foreach \i in {2,3,4} {
    \fill[black] (v\i) circle (3pt);
    \fill[black] (w\i) circle (3pt);
}
\fill[magenta] (v1) circle (3pt);
\fill[magenta] (w1) circle (3pt);

\node[above left, color=magenta]  at (v1) {$v_1$};
\node[below left]  at (v2) {$v_2$};
\node[below right] at (v3) {$v_3$};
\node[above right] at (v4) {$v_4$};

\node[left, color=magenta]  at (w1) {$w_1$};
\node[below] at (w2) {$w_2$};
\node[right] at (w3) {$w_3$};
\node[above] at (w4) {$w_4$};
\end{tikzpicture}

\medskip
\textbf{Circumscribed flower \(CF_4\)}

\[
\dim(E)>\frac{d+4}{2}
\]

\vspace{0.4em}
\end{minipage}

\\
\hline
\end{tabular}

\caption{Examples of some of the pinned \(2\)-simplicial complexes we can handle. Magenta vertices indicate the pinned vertex set \(\mathcal P\). The displayed dimensional conditions are sufficient to guarantee positive Lebesgue measure of the corresponding pinned area vector set $(A_{\Delta})^G_{pins}(E)$.}
\label{fig:pinned-simplicial-showcase}
\endgroup
\end{figure}
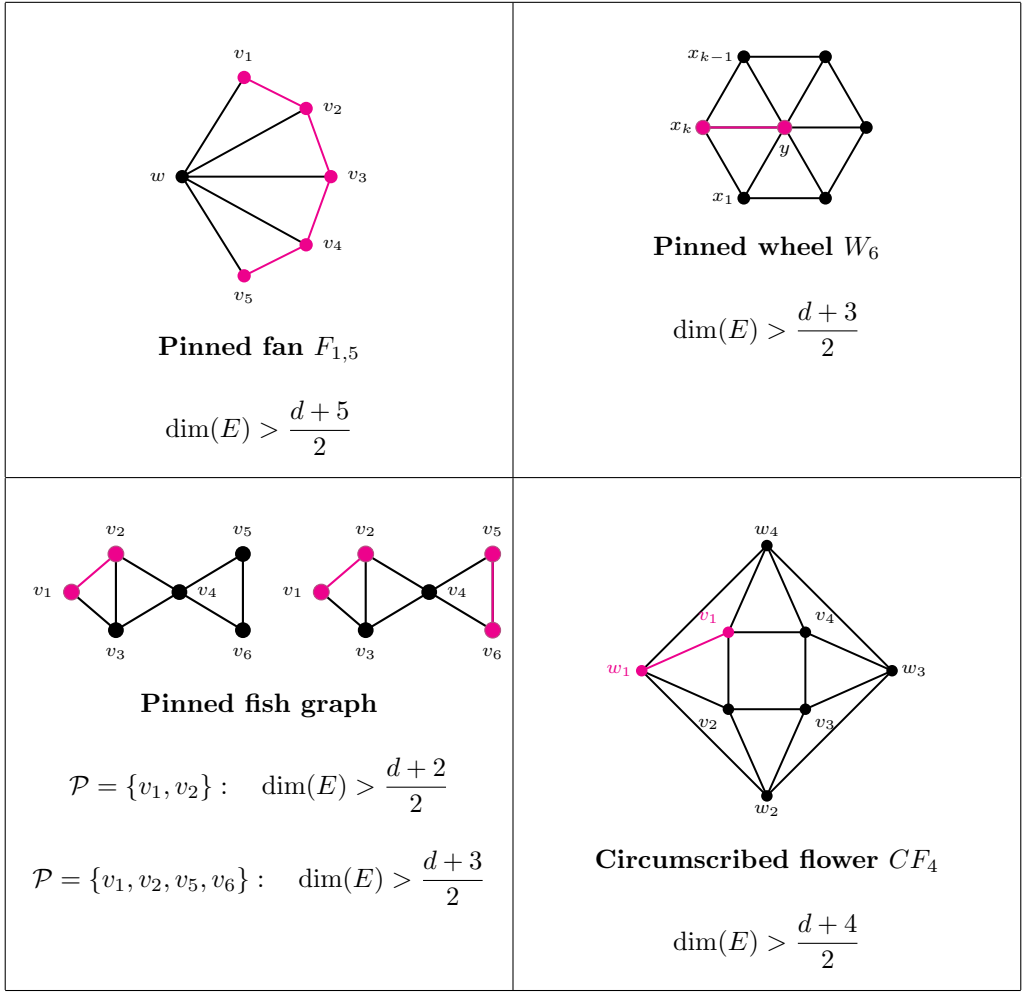

Due to interest in the original base configuration of a doubly pinned triangle, we record our results for that configuration. While Shmerkin and Yavicoli obtained a dimensional threshold of $1$ in $\mathbb{R}^2$ for a doubly pinned triangle and a slicing argument using their result obtains the threshold $2$ in $\mathbb{R}^3$ we obtain the best threshold known in $\mathbb{R}^d$, $d\geq 4$. We further record our result when $d=3$, as in that case we obtain a quantification on the number of good pins, which is not clear when using the slicing method. This result is a particular case of Theorem \ref{thm: k pinned fans}. In the statement below $\mu_{E_i}$ represents the restriction of $\mu$ to $E_i$ (see Definition \ref{def: restrictedmeasure}).

\begin{cor} 
   Suppose $d\geq 3$. Let $E\subset \R^d$ be a compact set with $\dim(E)>\frac{d+2}{2}$. Take $\mu$ an $s$-Frostman measure on $E$, with $s>\frac{d+2}{2}$. 
    Let $E_1,E_2,E_3$ be pairwise separated compact subsets of $E$ with $\mu(E_i)>0$ for all $i=1,2,3$. Then, for     $\mu_{E_1}\times\mu_{E_2}$ almost every 
$(x_1,x_2)\in E_1\times E_2$, one has
    \[
    \mathcal{L}^{1}((A_{\Delta})_{x_1,x_2}(E_3):=\{\A(x_1,x_2,x_3)\colon x_3\in E_3\})>0.
    \]
\end{cor}

The dimensional threshold $1$ of Shmerkin and Yavicoli in $\mathbb{R}^2$ is sharp, as can be seen by considering a line. The same sharpness example works in higher dimensions, but the gap between the example and our results grows with the dimension. Recently, Gaitan Montejo and the fourth listed author \cite{GMP26} had some success in extending the results of Shmerkin and Yavicoli for simplices to intermediate dimensions, but no improvement was obtained in the particular case of triangles.

There are some interesting $2$-simplicial configurations which are not covered by the two results above. The banana graph in Figure \ref{fig:banana} is $3$-degenerate, so positive measure of the edge set $\Delta^G(E)$ holds if $\dim(E)>\frac{d+3}{2}$. Say we include all the $7$ triangles in the picture as part of its  $2$-simplicial structure that defines the area set $(A_{\Delta})^G(E)$, then we have the property that any injective map from the set of triangles $\mathcal{E}_2$ to the set of edges $\mathcal{E}_1$ will have at least one triangle with all its three edges selected, so Theorem \ref{thm: mainJacobianthm} does not apply. 

\begin{figure}
\begin{tikzpicture}[scale=1.5]

\coordinate (A) at (0,0.3);        
\coordinate (B) at (1,0.3);        
\coordinate (C) at (0.6,0.6);    
\coordinate (F) at (2,0.3);        
\coordinate (G) at (3,0.3);        
\coordinate (H) at (2.4,0.6);    
\coordinate (D) at (1.5,1.6);    
\coordinate (E) at (1.5,-0.9);   


\draw[thick] (F) -- (G) -- (H) -- cycle;
\draw[thick] (F) -- (D); \draw[thick] (G) -- (D); \draw[thick] (H) -- (D);
\draw[thick] (F) -- (E); \draw[thick] (G) -- (E); \draw[thick] (H) -- (E);

\fill[black] (D) circle (1.5pt);   
\fill[black] (E) circle (1.5pt);   
\fill[black] (F) circle (1.5pt);    
\fill[black] (G) circle (1.5pt);    
\fill[black] (H) circle (1.5pt);    

\node at (1.5, 1.75) {$v_1$};
\node at (1.5, -1.05) {$v_2$};
\node at (1.8, 0.3) {$v_3$};
\node at (3.2, 0.3) {$v_4$};
\node at (2.2, 0.65) {$v_5$};

\end{tikzpicture}
\caption{The banana graph.}
\label{fig:banana}
\end{figure}
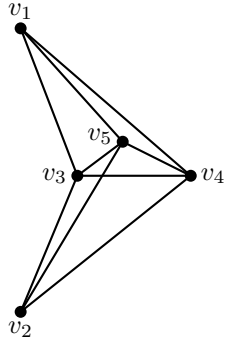

We can, however, prove that length-area transferal holds for this graph by analyzing the Jacobian of the map that sends edges to areas directly. For more details, see Proposition \ref{prop:banana-transferal} where we compute and analyze the relevant Jacobian. Note that if we don't include the $2$-simplex $(v_3,v_4,v_5)$ in $\mathcal{E}_2$, and only compute the areas of the $6$ triangles that comprise the surface of the banana, then the hypothesis of Theorem \ref{thm: mainJacobianthm} is satisfied using a convenient edge choice function, so then the associated area set has positive measure in $\R^6$ if $\dim(E)>\frac{d+3}{2}$.

 In all the examples discussed above, the passage from positive-measure
results for $\Delta^G(E)$ to the corresponding results for
$(A_{\Delta})^G(E)$, and similarly for their pinned counterparts, relied on the following input: if $G$ is a pinned or unpinned graph which is
$k$-admissible in the sense introduced in our previous work \cite{BFOPRpaperI}, then for any $E\subset \R^d$ satisfying the dimensional
threshold
\[
\dim_{\mathcal H}(E)>\frac{d+k}{2},
\]
$\Delta^G(E)$ has positive Lebesgue measure. Such a threshold only gives information when $1\leq k<d$, and it is intrinsically related to the threshold obtained by \cite{IPPS22} for $k$-stars, which we will recall in Theorem \ref{thm:k starinIPPS}.

In an upcoming work of the first and third authors, in collaboration with
M. Pasquariello \cite{BOP2026}, the best-known thresholds for $k$-stars are improved for $1\leq k<d$.
More precisely, the threshold $\frac{d+k}{2}$ can be replaced by
\[
\alpha_{+}(d,k)
:=
\frac{d^2+dk+k}{2d+1}
=
\frac{d+k-1}{2}
+\frac14
+\frac{2k+1}{4(2d+1)}.
\]
These improved thresholds for $k$-stars immediately upgrade the dimensional
thresholds in the results above.

The Jacobian method does not necessarily have to rely on the threshold
coming from the $k$-admissibility of the graph. Rather, once the
length-area transfer has been verified for a given simplicial complex, one
can use any available positive-measure threshold for $\Delta^G(E)$. A useful example where checking $k$-admissibility is not helpful is given by
the complete graph $K_{d+1}$ (which one also can call a $d$-dimensional simplex in \cite{GILP15}), with vertices sampled from
$E\subset \mathbb{R}^d$, in the cases $d=3,4$ (for which we will check the length-area transferal holds). Indeed, this corresponds to
the $d$-degenerate regime, and no positive-measure results for $d$-stars in
$\mathbb{R}^d$ are possible with thresholds smaller than $d$ \cite[Remark 7]{BFOPRpaperI}.
Thus, the results of \cite{BFOPRpaperI} do not provide information for these configurations, but one can use the thresholds for $k$-dimensional simplices from \cite{GILP15} instead. We will provide more details on this example in Section \ref{sec: jacobianmethod}.

We don't know of any examples of simplicial complexes where there is an injective way of picking an edge in each triangle and length-area transferal fails. That motivates us to conjecture that an injective edge function should be sufficient for length-area transferal (see definition \ref{def:triangleedgechoice} for the definition of triangle edge choice function).

\begin{conj}\label{conj:edgeareatransferal}
    Let $G=(\mathcal{V},\mathcal{E})$ be a simplicial complex. Then length-area transferal holds for $G$ if and only if there exists an injective triangle edge choice function.
\end{conj}

Interestingly, it turns out that the methods of this section are intimately related to structural rigidity theory, with open conjectures from that field (see \cite{LNPR}) also appearing in this setting. Our Theorem \ref{thm: mainJacobianthm} is merely a sufficient condition, but it is robust enough to handle a wide range of common $2$-dimensional simplicial complexes.

When the Jacobian method is combined with distance-graph results coming from graph degeneracy, it embodies the following principle: if a complicated area configuration can be expressed through a simpler distance configuration, then positive-measure results for the latter can be transferred to the former. Once length-area transferal has been verified for a given simplicial complex, however, the method is not tied to thresholds coming from $k$-admissibility. One may instead use any available Falconer-type positive-measure result for the underlying distance graph.

At the same time, the strength of the Jacobian method is limited by the best known distance or simplex thresholds for the relevant graph. In the plane, for instance, one can transfer the $8/5$ threshold for positive measure of unpinned triangles in the plane coming from the group-action method in \cite{GILP15}, but this does not reach the sharp threshold $1$ known for doubly pinned triangle areas by the theorem of Shmerkin and Yavicoli. For this reason, we now turn to a different set of arguments in the planar setting, based on projection theory and Fubini-type gluing. These arguments recover sharp planar thresholds for triangle-area configurations and provide model results for hypergraph area problems beyond the range of the Jacobian method.

\subsection{Projection methods in the plane}

In \cite{SY25}, Shmerkin and Yavicoli showed that if $E\subseteq \R^2$ is a compact set with $\dim(E)>1$ then there is a pair $(x_1,x_2)\in E\times E$ such that $$(A_\Delta)_{x_1,x_2}(E):= \{\A(x_1,x_2,x_3)\colon x_3\in E\}$$ has positive Lebesgue measure. Their proof relies on the geometric fact that, once
two vertices of a triangle are fixed, the area is proportional to the height
over the pinned edge. Thus, doubly pinned triangle areas can be studied
through one-dimensional orthogonal projections in directions perpendicular to
pinned edges. Their proof combines Marstrand's slicing theorem and positive measure of orthogonal projections for almost every direction.

As our first result, we refine their result by guaranteeing an abundance of pairs of pins.

\begin{thm}\label{thm:strengthenedSY}
    Let $E\subset \R^2$, with $\dim(E)>1$ and let $\mu$ be an $\alpha$-Frostman measure on $E$ for $\alpha>1$. Take disjoint compact subsets $E_1,E_2,E_3$ with $\mu(E_i)>0$, $i=1,2,3$. Then, for $\mu_{E_1}\times \mu_{E_2}$ a.e. pair $(x_1,x_2)\in E_1\times E_2$ one has 
    $$\mathcal{L}^{1}\left(\{\A(x_1,x_2,x_3)\colon x_3\in E_3\}\right)>0$$
\end{thm}

The result above follows from an even stronger triangle-area building block, yielding $L^2$ regularity for the pushforward of the Frostman measure in $E$ via orthogonal projections in relevant directions. That is stated precisely in Theorem \ref{thm:trianglebuildingblockintheplane}, and the proof will require two main ingredients: $L^2$ regularity of orthogonal projection estimates, and the key new ingredient of Orponen's $L^p$ estimates for radial projections (recall its statement in Theorem \ref{thm:OrponenL^pprojection}).
In our previous paper \cite{BFOPRpaperI}, where $k$-stars served as building blocks for more general pinned distance graph configurations, an analogous use of $L^2$-type strengthening was also needed. The $L^2$ strength of the $k$-star building blocks was important to prune sets during the induction process without losing the positive measure properties of the $k$-stars that came earlier in the graph-building process.

Using such a building block, one can get to results for triangle chains sharing vertices or edges. More precisely, let $TC_k^{edge}$ stand for the $2$-simplicial complex associated with triangle chains of $k$-triangles attached by edges. See Figure \ref{fig:three-triangle-chain} for the case $k=3$.

\begin{figure}[h!]
\centering
\begin{tikzpicture}[
  v/.style = {circle, fill=black, inner sep=1.8pt},
  mv/.style = {circle, draw=magenta!70!black, fill=magenta, inner sep=2pt}
]
  \coordinate (v1) at (-2.4,0);
  \coordinate (v3) at (0,0);
  \coordinate (v5) at (2.4,0);
  \coordinate (v2) at (-1.2,1.5);
  \coordinate (v4) at (1.2,1.5);

  \draw[line width=0.8pt] (v1)--(v2)--(v3)--cycle;
  \draw[line width=0.8pt] (v2)--(v3)--(v4)--cycle;
  \draw[line width=0.8pt] (v3)--(v4)--(v5)--cycle;

  \node[mv, label=left:$v_1$] at (v1) {};
  \node[mv, label=above:$v_2$] at (v2) {};
  \node[v, label=below:$v_3$] at (v3) {};
  \node[v, label=above:$v_4$] at (v4) {};
  \node[v, label=right:$v_5$] at (v5) {};
\end{tikzpicture}
\caption{An edge-attached \(3\)-triangle chain $TC_3^{edge}$ with pins in vertices $v_1$ and $v_2$.}
\label{fig:three-triangle-chain}
\end{figure}
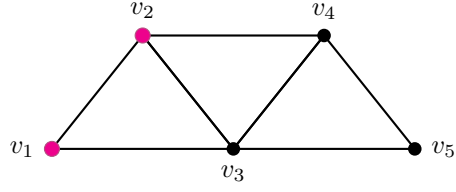

Given $E\subset \R^d$ and  $x_1, x_2, \dots ,x_{k+2}\in E$, define the area computing function for $k$-triangle chains
    $$A^{TC_k^{edge}}(x_1,x_2,\hdots, x_{k+2}) := \left( \A(x_1,x_2,x_3), \A(x_{2},x_{3},x_{4}), \hdots, \A(x_{k},x_{k+1},x_{k+2}) \right).$$
For fixed $x_1,x_2\in E$, also consider the pinned variant
    $$(A^{TC_k^{edge}})_{x_1,x_2}(x_3,\hdots, x_{k+2}):=A^{TC_k^{edge}}(x_1,x_2,\hdots, x_{k+2}).$$
    
\begin{thm}[Edge-attached triangle chains]\label{cor:edgetrianglechains}
 Let $k\geq 1$. For a compact set $E\subseteq \R^2$, consider the area set associated with the edge attached $k$-triangle chains in $E$ with the initial edge pinned, namely
    $$(A_{\Delta})^{TC^{edge}_k}_{x_1,x_2}(E)= \{(A^{TC_k^{edge}})_{x_1,x_2}(x_3,x_4,\dots, x_{k+2})\colon x_i\in E\text{ distinct },\,3\leq i\leq k+2\}.$$
If $\dim(E)>1$, then there exist $x_1\neq x_2\in E$ such that 
$\mathcal{L}^k\left((A_{\Delta})^{TC_k^{edge}}_{x_1,x_2}(E) \right)>0.$
\end{thm}

This result is sharp, as can be seen by considering a line, and improves on a result by Galo and McDonald \cite{GM22}, who obtained a dimensional threshold $\frac{3}{2}$.

It is possible to use the building block in Theorem \ref{thm:strengthenedSY} (without needing the $L^2$ strength in this case) to obtain results for area sets realized in chains of vertex-attached triangles in the plane, namely

\begin{thm}[Vertex-attached triangle chains]\label{cor: vertextrianglechain}
    Let $k\geq 1$. For a compact set $E\subseteq \R^2$ consider the area set associated with the vertex attached $k$-triangle chains in $E$ with first vertex pinned, namely
    $$ (A_{\Delta})_{x_1}^{TC_k^{vtx}}(E):= \lbrace (\textnormal{Area}(x_{2i-1},x_{2i},x_{2i+1}))_{i=1}^{k} \colon x_2,x_3,\dots ,x_{2k+1}\in E \text{ distinct }\rbrace.  $$
    If $\dim(E)>1$ then there exists $x_1\in E$ such that $\mathcal{L}^{k}\left((A_{\Delta})_{x_1}^{TC_k^{vtx}}(E)\right)>0$.
\end{thm}

This result is sharp, as can be seen by considering a line, and improves on a result of Greenleaf, Iosevich and Taylor \cite{GIT25}, where they obtained the dimensional threshold $\frac{5}{3}$ for the stronger conclusion of non-empty interior.

We finally note that we can mix and match connecting triangles on vertices or edges. To showcase that we highlight a particular example that comes from a graph we call the \emph{fish graph}, see Figure \ref{fig:fish}. Here, the relevant set for pinned vertices $x_1, x_2$ is

$$ (A_{\Delta})^{\text{fish}}_{x_1,x_2} = \lbrace (\text{Area}(x_1,x_2,x_3),\text{Area}(x_2,x_3,x_4),\text{Area}(x_4,x_5,x_6)) : x_3,x_4,x_5,x_6\in E \rbrace .$$

\begin{center}
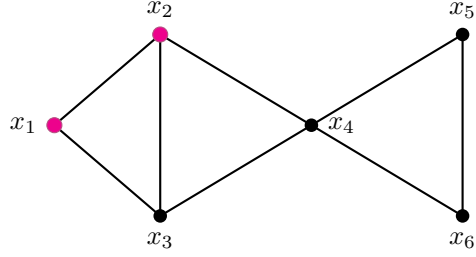
\begin{figure}[h!]
\begin{tikzpicture}[
  v/.style = {circle, fill=black, inner sep=1.8pt},
  mv/.style = {circle, draw=magenta!70!black, fill=magenta, inner sep=2pt}
]
  \coordinate (x3) at (0,0);
  \coordinate (x1) at (2,1.2);
  \coordinate (x2) at (2,-1.2);
  \coordinate (x6) at (-2,1.2);
  \coordinate (x4) at (-2,-1.2);
  \coordinate (x5) at (-3.4,0);

  \draw[line width=0.8pt] (x1)--(x2)--(x3)--(x1);     
  \draw[line width=0.8pt] (x5)--(x6)--(x4)--(x5);     
  \draw[line width=0.8pt] (x6)--(x3)--(x4);           

  \node[v, label=right:$x_4$] at (x3) {};
  \node[v, label=above:$x_5$] at (x1) {};
  \node[v, label=below:$x_6$] at (x2) {};
  \node[mv,label=above:$x_2$] at (x6) {};
  \node[v, label=below:$x_3$] at (x4) {};
  \node[mv,label=left:$x_1$] at (x5) {};
\end{tikzpicture}
\caption{The fish graph with pins at $x_1$ and $x_2$. This example encompasses both cases of triangles that share a common edge and triangles that merely share a single common vertex}
\label{fig:fish}
\end{figure}
\end{center}

This configuration allows us to showcase our methods, including triangles attached at both vertices and edges, as well as both proof methods. We obtain the following theorem, which follows from Theorem \ref{thm: fish} for $d\geq 3$ and whose proof in the case $d=2$ will be given right after Corollaries \ref{cor:edgetrianglechains} and \ref{cor: vertextrianglechain}.

\begin{thm}\label{thm: fishAllD}
Suppose $E\subset\mathbb{R}^d$, $d\geq 2$, is compact with $\dim(E)>1$ if $d=2$ and $\dim(E)>(d+2)/2$ if $d\geq 3$. Then there exist points $x_1,x_2\in E$ such that
\[
\mathcal{L}^3((A_{\Delta})^{\text{fish}}_{x_1,x_2})>0.
\]
\end{thm}

\subsection*{Organization}
Here is an overview of the organization of the paper. In Section \ref{sec: preliminaries} we introduce some preliminaries including the notion of a simplicial complex which records not only triangles but also their sides. In Section \ref{sec: jacobianmethod} we develop the Jacobian method that allows us to transfer results on distance graphs to area or volume vectors determined by simplicial complexes. In Section \ref{sec: projectionmethod} we refine a theorem of Shmerkin and Yavicoli quantifying the number of pins for establishing a threshold for double pinned areas of triangles in the plane. We use this building block to develop results on chains of areas of triangles connected either by vertices or edges inspired by the inductive Fubini argument strategy of Ou and Taylor and record these inductive strategies in Appendix \ref{structuralThm}.

\subsection*{Acknowledgement}
Y.O. is supported in part by NSF DMS-2142221 and NSF DMS-2055008. The authors would like to thank Alex Iosevich, Matt Larson, Jill Pipher and Maya Sankar for many helpful and inspiring conversations on the topic of this paper.

\vskip1.5em

\section{Preliminaries on hypergraphs, simplicial complexes, and volumes}\label{sec: preliminaries}
Given a graph $G=(\mathcal{V},\mathcal{E})$, we are interested in understanding the area vectors determined by certain triangles appearing in the graph. A triangle is determined by multiple vertices, which puts us naturally into the world of \textbf{hypergraphs}. For the reader's convenience, we review the basic definitions on hypergraphs here. Throughout, we will assume all sets are finite as we only deal with finite hypergraphs in this paper.
\begin{defn}
    A hypergraph is a pair $(\mathcal{V},\E)$ where $\mathcal{V}$ is a set we call the vertices, and $\E \subset 2^\mathcal{V}$ is a collection of subsets of $\mathcal{V}$, which we call the hyperedges of the graph. 
\end{defn}

In the case that every $e\in \E$ has cardinality 2, we recover the usual definition of a graph. Similarly, an edge $e$ with cardinality 3 is a triple of vertices, which determines a triangle. Note that, under this definition, it is possible to have a hyperedge $e$ included in $\E$ while certain proper subsets (which may correspond to edges, for instance) are not included in $\E$. We want to rule out this behavior, so we will work with simplicial complexes instead.
\begin{defn}
    A hypergraph $(\mathcal{V},\E)$ is a simplicial complex if whenever $e\in \E$ and $f\subset e$ with $|f|>1$ (where $|\cdot |$ denotes cardinality), we have that $f\in \E$.
\end{defn}
This definition guarantees, among other things, that whenever a triangle is included in the data of a hypergraph, then all the edges of that triangle are also included. The restriction that $|f|>1$ is for convenience so as to avoid listing vertices again in the hyperedge set.
\begin{defn}
    In a simplicial complex $(\mathcal{V},\E)$, we say that $e\in \E$ is a $k$-simplex if $|e|=k+1$. If we wish, we may stratify the simplicial complex into the tuple $(\mathcal{V},\E_1,\ldots,\E_m)$ where for each $1\leq i\leq m$, $e\in \E_i$ if and only if $e$ is an $i$-simplex.
\end{defn}
In particular, the closure properties of a simplicial complex guarantee that if a $k$-simplex is included as a hyperedge, then all its lower dimensional faces are also included as hyperedges.

As an example, consider the following hypergraph. \\
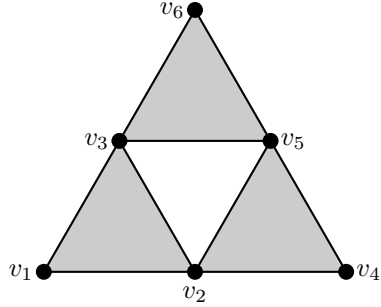
\begin{figure}[h]
    \centering
\begin{tikzpicture}[scale=2]
    \coordinate (A) at (0, 0);
    \coordinate (B) at (1, 0);
    \coordinate (C) at (0.5, {sqrt(3)/2});
    
    \coordinate (D) at (2, 0);
    \coordinate (E) at (1.5, {sqrt(3)/2});
    
    \coordinate (F) at (1, {sqrt(3)});

    \filldraw[fill=gray!40, draw=black, thick] (A) -- (B) -- (C) -- cycle; 
    \filldraw[fill=gray!40, draw=black, thick] (B) -- (D) -- (E) -- cycle; 
    \filldraw[fill=gray!40, draw=black, thick] (C) -- (E) -- (F) -- cycle; 
    \foreach \point in {A, B, C, D, E, F} {
        \fill[black] (\point) circle (1.5pt);
    }
    \node[black,shift={(-0.3,0)}]      at (A) {$v_{1}$};
    \node[black,shift={(0,-0.3)}]      at (B) {$v_{2}$};
    \node[black,shift={(-0.3,0)}]      at (C) {$v_{3}$};
    \node[black,shift={(0.3,0)}]      at (D) {$v_{4}$};
    \node[black,shift={(0.3,0)}]      at (E) {$v_{5}$};
    \node[black,shift={(-0.3,0)}]      at (F) {$v_{6}$};
\end{tikzpicture}
\end{figure} \\
Here, the shaded triangles refer to the $2$-simplices in the hyperedge set. For this hypergraph, we have that $\E_2=\{\{v_1,v_2,v_3\},\{v_2,v_4,v_5\},\{v_3,v_5,v_6\}\}$; note that the triangle with vertices at $v_2,v_3,v_5$ is not included in the hyperedge set, even though all the edges making it up are in $\E$.

In particular, we will find it useful to recall formulas for the volumes of $k$-simplices in terms of the side lengths (and of course, these side lengths are determined by the vertices). For $2$-simplices, this is the well-known Heron's formula, which says that the area of a triangle can be computed in terms of the side lengths. Let $x_1,x_2,x_3$ be three vertices of a triangle in $\R^d$. Let $s=|x_1-x_2|$, $t_1=|x_1-x_3|, t_2=|x_2-x_3|$. Then Heron's formula implies
\begin{equation}\label{eq heron area}
16\text{ Area}^2(x_1,x_{2},x_3)=((t_1+t_2)^2-s^2)(s^2-(t_1-t_2)^2)=:\varphi_{s}(t_1,t_{2})\end{equation}

More generally, there are analogues of Heron's formula for computing $k$-volumes of higher-dimensional $k$-simplices, known as Cayley--Menger formulas. Given $N+1$ points $$x_1,x_2,\ldots, x_{N+1}\in\R^{k},$$ where $k\geq N$, let $B$ be the $(N+1)\times (N+1)$ matrix whose $(i,j)$th entry is $b_{ij}=|x_{i}-x_{j}|^2$ and let $\tilde{B}$ be the $(N+2)\times (N+2)$ block matrix
\[
\tilde{B} = \begin{bmatrix}
    B & \mathbf{1} \\
    \mathbf{1}^{T} & 0
\end{bmatrix}
\]

It is well known (see e.g. \cite[Section 2]{CMSource}) that the formula for the volume of an $N$-simplex with vertices $x_1,\ldots,x_{N+1}$ is given by
\begin{equation}\label{eq Cayley Menger}
\text{Vol}(x_1,\ldots,x_{N+1})^2=\frac{1}{2^N(N!)^2}|\det\tilde{B}|
\end{equation}
For our applications, we will not be too concerned with the exact formula; the most relevant thing is that the squared volume can be computed as a polynomial function of the squared side lengths, and hence as a polynomial function of the side lengths, of the \(N\)-simplex in question.

For simplicity, we will focus primarily on the case of areas of $2$-simplices. Given a $2$-dimensional simplicial complex $G=(\mathcal{V},\mathcal{E})$, we denote its area vector set as
\begin{equation}\label{eq area vector set}
(A_{\Delta})^G(E):=\{(\text{Area}(x_i,x_j,x_k)^2)_{(v_i,v_j,v_k)\in\mathcal{E}_2}:x_1,\ldots,x_{|\mathcal{V}|}\in E\}
\end{equation}

In the pinned scenario, where we add pins at the vertices in $\mathcal{P}=\{v_1,\dots,v_m\}\subset \mathcal{V}$, the pinned counterpart becomes
$$(A_{\Delta})^G_{x_1,\dots ,x_m}(E):=\{(\text{Area}(x_i,x_j,x_k)^2)_{(v_i,v_j,v_k)\in\mathcal{E}_2}:x_{m+1},\ldots,x_{|\mathcal{V}|}\in E\}
$$
as long as this choice of pins doesn't fully pin any $2$-simplex $T\in \mathcal{E}_2$, that is, if for all $T=\{v_i,v_j,v_k\}\in \mathcal{E}_2$ we have $T$ not contained in $\mathcal{P}$. In the case where some $2$-simplices are fully pinned, we just have to be careful not to include them in the area set definition (since the area won't vary for that triangle).

Note that the difference between studying area and its square is unimportant for dimensional considerations as the map $t\mapsto t^2$ is bi-Lipschitz onto its image when restricted to a compact domain of the form $[\epsilon,M]$. We will compose with other suitably nice maps in the sequel as well in order to simplify the algebraic expressions that will appear.

\section{Jacobian Method}\label{sec: jacobianmethod}

Throughout this section, we assume that $d\geq 3$. The building block, that drives our results for general pinned graphs of distances, are so called $k$-stars. Let $k\geq 1$ be an integer. By a $k$-star we will mean a graph in $(k+1)$ vertices in which a central vertex is connected by edges to any other vertex, and these are the only vertices in the graph. We will denote such a graph by $S_k$ (see the Figure \ref{fig: stargraph} for $S_7$).

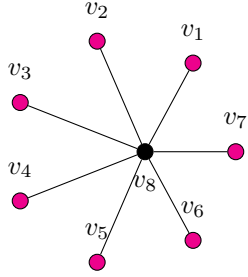
\begin{figure}[h]
    \centering
    \begin{tikzpicture}[scale=1.5, 
    every node/.style={circle, draw, minimum size=6pt, inner sep=0pt},
    label distance=1mm]

\node[fill=black, label=below:$v_8$] (C) at (0.2,0) {};

\foreach \i in {1,...,7} {
    \node[fill=magenta, label=above:$v_{\i}$] 
        (A\i) at ({cos(360/7*\i)}, {sin(360/7*\i)}) {};
    \draw (C) -- (A\i);
}

\end{tikzpicture}
    \caption{$S_7$, the $7$-star graph, with pins at each leaf.}
    \label{fig: stargraph}
\end{figure}

\begin{defn}\label{def: pinnedkstarsinE} Let $E\subset \R^d$ be a compact set.
The set of $k$-stars generated by $E$ with pins at $x_1,x_2,\dots,x_k$ is given by 
$$\Delta^{k-star}_{x_1,\dots,x_k}(E):=\{(|x_1-y|,|x_2-y|,\dots,|x_k-y|)\colon y\in E\}\subset \R^k.$$
\end{defn}

In \cite{IPPS22} they showed that if $E\subset \R^d$ is a compact set with $\dim(E)>\frac{d+k}{2}$ then there exists $x_1,x_2,\dots, x_k\in E$ such that $\Delta^{k-star}_{x_1,\dots,x_k}(E)$ has positive $k$-dimensional Lebesgue measure. In fact, they proved something even stronger, which we will recall below.

First, let us introduce a couple of definitions.

\begin{defn} Given a compact set $E\subset \R^d$ an $s$-Frostman measure $\mu$ on $E$ is a nonzero compactly supported on $E$ finite Borel measure satisfying that 
$$\mu(B(x,r))\lesssim r^s,\, \text{ for all }x\in \R^d\text{ and }r>0.$$
\end{defn}
By Frostman's lemma, such measures exist for all $s<\dim E$.

\begin{defn}\label{def: restrictedmeasure}
    Given an $s$-Frostman measure $\mu$ supported on a compact $E\subset \R^d$ and $F$ a compact subset of $E$, denote by $\mu_F$ the restriction of $\mu$ to $F$. That is, $\mu_F(A)=\mu(A\cap F)$ for any Borel set $A\subset \R^d$. When $\mu(F)>0$, $\mu_F$ is still an $s$-Frostman measure.

\end{defn}

\begin{thm}[\cite{IPPS22}]\label{thm:k starinIPPS}

Let $E\subset \R^d$ be a compact set with $\dim(E)>\frac{d+k}{2}$. Take $\mu$ an $s$-Frostman measure on $E$, with $s>\frac{d+k}{2}$. 
    Let $E_1,E_2,\dots E_{k+1}$ be pairwise separated compact subsets of $E$ with $\mu(E_i)>0$ for all $i=1,2,\dots ,k+1$, then for $\mu_{E_1}\times \mu_{E_2}\times \cdots \times\mu_{E_{k}}$ almost every     $(x_1,x_2,\dots ,x_k)\in E_1\times E_2\times \dots \times E_k$, one has that the pushforward measure
    $$(d^{k-star}_{x_1,x_2,\dots ,x_k})_{*}(\mu_{E_{k+1}})\text{ belongs to } L^2(\R^k),$$
    where $d^{k-star}_{x_1,x_2,\dots ,x_k}(y):=(|x_1-y|,|x_2-y|,\dots ,|x_k-y|)$, and, consequently,
$$\mathcal{L}^k\left(\Delta^{k-star}_{x_1,\dots,x_k}(E_{k+1})\right)>0.$$
\end{thm}

Using Theorem \ref{thm:k starinIPPS}, we will show that we can find positive measure sets of area vectors corresponding to fan graphs; this can also be viewed as a blueprint for approaching more general graphs.

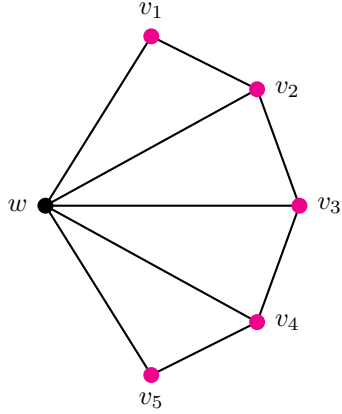
\begin{figure}[h]
\centering
\begin{tikzpicture}[scale=1.4, >=latex]

\tikzset{
  vtx/.style = {circle, inner sep=0pt, minimum size=6pt},
  edge/.style = {thick}
}

\coordinate (w)  at (0,0);

\coordinate (v1) at (1,1.6);
\coordinate (v2) at (2,1.1);
\coordinate (v3) at (2.4,0);
\coordinate (v4) at (2,-1.1);
\coordinate (v5) at (1,-1.6);

\draw[edge] (w) -- (v1);
\draw[edge] (w) -- (v2);
\draw[edge] (w) -- (v3);
\draw[edge] (w) -- (v4);
\draw[edge] (w) -- (v5);

\draw[edge] (v1) -- (v2);
\draw[edge] (v2) -- (v3);
\draw[edge] (v3) -- (v4);
\draw[edge] (v4) -- (v5);

\node[vtx, fill=black, label=left:$w$] at (w) {};

\node[vtx, fill=magenta, label=above:$v_1$] at (v1) {};
\node[vtx, fill=magenta, label=right:$v_2$] at (v2) {};
\node[vtx, fill=magenta, label=right:$v_3$] at (v3) {};
\node[vtx, fill=magenta, label=right:$v_4$] at (v4) {};
\node[vtx, fill=magenta, label=below:$v_5$] at (v5) {};

\end{tikzpicture}
\caption{The fan graph $F_{1,5}$ with pinned vertices at $v_1,\dots,v_5$.}
\label{fig:fanoftraingles}
\end{figure}

For $k\ge 2$, let $F_{1,k}$ denote the fan graph with $k+1$ vertices, including one distinguished vertex $w$ of degree $k$. Formally, $F_{1,k}$ is defined as the graph join of $K_1$ (the singleton graph) and $P_k$ (the path graph with $k$ vertices), where the graph join consists of taking the union of the vertex sets and adding an edge from every vertex in $K_1$ to every vertex in $P_k$. The graph $F_{1,k}$ has $k-1$ triangles, whose vertices are always triples of the form $(v_i,v_{i+1},w)$ where $1\le i\le k-1$. We will study area sets of fan graphs where every vertex except $w$ is pinned (see Figure \ref{fig:fanoftraingles} for $k=5$). Precisely, for any compact set $E\subset \mathbb{R}^d$ and $x_1,x_2,\dots, x_k\in E$, define
\[
(A_{\Delta})^{F_{1,k}}_{x_1,\ldots,x_k}(E):=\{(\text{Area}(x_1,x_2,y),\text{Area}(x_2,x_3,y),\ldots,\text{Area}(x_{k-1},x_k,y)):y\in E\}.
\]
We will want to compute the areas in terms of the three side lengths of the triangle, using Heron's formula. It will be convenient to work with the square of the area rather than the area; notice that when the areas are upper and lower bounded (which is the case when the vertices are taken to be in separated compact subsets of Euclidean space), then the map carrying areas to their squares is bi-Lipschitz. Let $t_i=|y-x_i|$, $1\leq i\leq k$ and let $s_i=|x_{i+1}-x_i|$, $1\leq i\leq k-1$. Then Heron's formula implies
\[
16\text{ Area}^2(x_i,x_{i+1},y)=((t_i+t_{i+1})^2-s_i^2)(s_i^2-(t_i-t_{i+1})^2)=:\varphi_{s_i}(t_i,t_{i+1}).
\]

Our first task in this section is to show that Theorem \ref{thm:k starinIPPS} transfers to this setting to give lots of area vectors, even when all of the non-distinguished vertices in the fan graph are pinned.

\begin{thm} \label{thm: k pinned fans}
   Suppose $d\geq 3$ and $2\leq k<d$. Let $E\subset \R^d$ be a compact set with $\dim(E)>\frac{d+k}{2}$. Take $\mu$ an $s$-Frostman measure on $E$, with $s>\frac{d+k}{2}$. 
    Let $E_1,E_2,\dots E_{k+1}$ be pairwise separated compact subsets of $E$ with $\mu(E_i)>0$ for all $i=1,2,\dots ,k+1$, then for $\mu_{E_1}\times \mu_{E_2}\times \cdots \times\mu_{E_{k}}$ almost every $(x_1,x_2,\dots ,x_k)\in E_1\times E_2\times \dots \times E_k$, one has that
    \[
    \mathcal{L}^{k-1}((A_{\Delta})^{F_{1,k}}_{x_1,\ldots,x_k}(E_{k+1}))>0.
    \]
\end{thm}

\begin{proof}
From Theorem \ref{thm:k starinIPPS} we know that for $\mu_{E_1}\times \mu_{E_2}\times \dots \times\mu_{E_k}$ almost every $(x_1,x_2,\dots ,x_k)\in E_1\times E_2\times \dots E_k$, it holds that 
$$\mathcal{L}^k\left(\Delta^{k-star}_{x_1,\dots,x_k}(E_{k+1})\right)>0.$$

We claim that any such $(x_1,x_2,\dots ,x_k)$
satisfy the desired property that
   $$ \mathcal{L}^{k-1}((A_{\Delta})^{F_{1,k}}_{x_1,\ldots,x_k}(E_{k+1}))>0
    .$$ Indeed, for these fixed $x_1,x_2,\dots, x_k$ let $s_j:=|x_j-x_{j+1}|$, $1\leq j\leq k-1$, and consider the mapping
    \[
    \Phi:\Delta^{k-star}_{x_1,\dots,x_k}(E_{k+1})\rightarrow \mathbb{R}^{k},\qquad (t_1,\ldots,t_k)\mapsto (\varphi_{s_1}(t_1,t_2),\ldots,\varphi_{s_{k-1}}(t_{k-1},t_k),t_k),
    \]
    where we have parameterized the domain via $t_j=|x_j-y|$, $y\in E_{k+1}$. We want to show that $\Phi$ is locally a bi-Lipschitz mapping onto its image. Note that the derivative mapping $D\Phi$ is upper triangular, with the $j$th diagonal entry being given by a nonzero polynomial of $t_j,t_{j+1}$ for $j<k$. The final diagonal entry of $D\Phi$ is the constant function 1. Thus, we see that the determinant $|D\Phi|$ vanishes on a union of $k-1$ many varieties, i.e. on some set $S\subset \mathbb{R}^k$ with $\mathcal{L}^k(S)=0$. Since $\mathcal{L}^k\left(\Delta^{k-star}_{x_1,\dots,x_k}(E_{k+1})\right)>0$, it follows that for $\epsilon>0$ sufficiently small, we can find an open neighborhood $U_{\epsilon}\supset S$ such that $|D\Phi|>\epsilon$ on $\Delta^{k-star}_{x_1,\dots,x_k}(E_{k+1})\setminus U_{\epsilon}$ and $\mathcal{L}^k(\Delta^{k-star}_{x_1,\dots,x_k}(E_{k+1})\setminus U_{\epsilon})>0$. 
    Since $\Phi$ has non-vanishing Jacobian on a subset of $\Delta^{k\text{-star}}_{x_1,\dots,x_k}(E_{k+1})$ of positive measure, we have
    \[
\mathcal{L}^k\big(\Phi(\Delta^{k\text{-star}}_{x_1,\dots,x_k}(E_{k+1}))\big)>0.
    \]
    Now, let $\pi_{k,k-1}:\mathbb{R}^k\rightarrow \mathbb{R}^{k-1}$ denote the projection $\pi_{k,k-1}(a_1,\ldots,a_k)=(a_1,\ldots,a_{k-1})$. Since $\mathcal{L}^k(\Phi(\Delta^{k-star}_{x_1,\dots,x_k}(E_{k+1})))>0$, it follows that $\mathcal{L}^{k-1}(\pi_{k,k-1}(\Phi(\Delta^{k-star}_{x_1,\dots,x_k}(E_{k+1}))))>0$. Unraveling the definitions, we deduce that $\mathcal{L}^{k-1}((A_{\Delta})^{F_{1,k}}_{x_1,\ldots,x_k}(E_{k+1}))>0$, as desired.
\end{proof}
One could run variants of this argument for other graphs. We choose not to do this here; instead, we will give a more general framework that will work for a wide variety of graphs in practice. The key idea is that we used the hypothesis of having an abundance of pinned length vectors for the $k$-star graph in order to deduce that we also had an abundance of pinned area vectors for the graph with the same dimensional threshold.

\subsection{Length to area transferal principles}
Here, we study the following question: given a simplicial complex $G=(\mathcal{V},\mathcal{E})$, what graph-theoretic conditions on $G$ guarantee that if the distance set $\Delta^G(E)$ has positive measure when $\dim E>s$, then the area set $(A_{\Delta})^G(E)$ also has positive measure whenever $\dim E>s$?

By Heron's formula, we have a surjective map $\Phi_E: \Delta^G(E)\rightarrow A_G(E)$ that we call the \textbf{length-to-area} map for $G$ associated to $E$. When we take $E$ to be all of Euclidean space, we may omit the subscript. It is given by
\[
\Phi(t_1,\ldots,t_{|\mathcal{E}_1|})=(\alpha(t_i,t_j,t_k))_{\{v_i,v_j,v_k\}\in\mathcal{E}}
\]
where $\alpha(x,y,z)$ denotes the area of the triangle with side lengths $x,y,z$. Composing with some nice enough maps, that is, studying instead $C\alpha^2(\sqrt{t_1},\sqrt{t_2},\sqrt{t_3})$, we may take
\[
\alpha(x,y,z)=-\frac{x^2}{2}-\frac{y^2}{2}-\frac{z^2}{2}+xy+xz+yz.
\]
We observe that 
$$\nabla\alpha(x,y,z)=(-x+y+z,x-y+z,x+y-z),$$
which provides a simple structure to the rows of $D\Phi$.

Since the domain of $\Phi_E$ has positive measure by assumption, it is enough to show that the derivative of $\Phi$ is surjective for some input $(t_1,\ldots,t_{|\mathcal{E}|})$, even if this input is not in $\Delta^G(E)$. Indeed, we know that $D\Phi$ is full rank somewhere if and only if the matrix-valued function $(D\Phi)(D\Phi)^T$ is invertible somewhere which happens if and only if its determinant (which is a polynomial in $t_1,\ldots,t_{|\mathcal{E}|}$) is not the 0 function. A nonzero polynomial can only vanish on a lower dimension algebraic variety which has measure 0; after excising this wherever it intersects $\Delta^G(E)$, we still have a positive measure subset of our domain on which $D\Phi_E$ is surjective, implying that the image of $\Phi_E$ has positive measure.

Thus, everything reduces to understanding when $D\Phi$ has linearly independent rows for some input. In fact, by Sard's theorem, this gives us a necessary condition.

\begin{prop}
    Let $\Phi$ denote the length-to-area map as above. Suppose that for every input in Euclidean space, the matrix $D\Phi$ does not have full row rank. Then the set $A_G(E)$ does not have positive Lebesgue measure.
\end{prop}
\begin{proof}
    By Sard's theorem, the Hausdorff dimension of the image of $\Phi$ has dimension equal to at most the maximum value of the row rank, which is less than the dimension of the codomain.
\end{proof}

\begin{defn}[Edge area transferal]\label{def: edgetransferal}
    We will say that edge area transferal holds for $G$ if $D\Phi$ has full row rank for some input.
\end{defn}
As a result of our discussion above, this condition guarantees that if we know that $\dim(E)>s$ implies that $\Delta^G(E)$ has positive measure, then we also deduce that $A_G(E)$ has positive measure. Finding conditions that are both necessary and sufficient for edge area transferal is more delicate. We start with a sufficient condition based on the simplicial complex structure that is easy to verify for a variety of examples.

\begin{prop}\label{prop suff transfer}
    Let $G=(\mathcal{V},\mathcal{E})$ be a simplicial complex with the property that for every $T\in \mathcal{E}_2$, there exists an edge $\{v,w\}\subset T$ with the property that for all $2$-simplices $T'\ne T$, we have that $\{v,w\}\not\subset T'$. Then length-area transferal holds for $G$. 
\end{prop}

\begin{proof}
    By our above discussion, we need to show that $D\Phi$ has linearly independent rows somewhere. For each triangle $T_j\in \mathcal{E}_2$, let $n(j)$ denote the index of the edge $\{v_j,w_j\}\subset T_j$ with the desirable property in the hypothesis of the proposition. Choose the input vector $\mathbf{t}$ whose $i$th component is given by $t_i=\begin{cases}
        1, & i=n(j)\text{ for some }j\in\{1,\ldots,|\mathcal{E}_2|\} \\
        0, & \text{otherwise}
    \end{cases}$. By construction, the $n(j)$th column of $D\Phi(\mathbf{t})$ is the negative of the $j$th standard basis vector; this gives $|\mathcal{E}_2|$ many linearly independent columns in $D\Phi(\mathbf{t})$. It follows immediately that the rank of $D\Phi(\mathbf{t})$ equals $|\mathcal{E}_2|$, which is the number of rows in that matrix.
\end{proof}
This allows us to already deduce some results for a number of families of graphs. Note that we can also formulate a pinned version of the above result with only minor modifications needed. We will say that an edge is pinned if both its vertices are in the pinned vertex set $\mathcal{P}\subset \mathcal{V}$. The basic idea is that we need the unshared edges of each triangle to be unpinned, i.e. each triangle has an edge shared with no other triangle, and at most one of the vertices of that edge is pinned.

\begin{prop}
    Let $G=(\mathcal{V},\mathcal{E},\mathcal{P})$ be a pinned simplicial complex with pin set $\mathcal{P}\subset\mathcal{V}$ which has the property that for every $T\in \mathcal{E}_2$, there exists an edge $\{v,w\}\subset T$ with the property that $\{v,w\}\not\subset \mathcal{P}$ and for all $2$-simplices $T'\ne T$, we have that $\{v,w\}\not\subset T'$. Then length-area transferal holds for $G$.
\end{prop}
\begin{proof}
    The proof is essentially the same as that of Proposition \ref{prop suff transfer}. Any pinned edges are fixed to take some arbitrary constant value as the length, and the derivative matrix $D\Phi$ has correspondingly fewer columns as a result. Let $L$ denote the maximum value of the lengths of pinned edges and for each triangle $T_j\in \mathcal{E}_2$, let $n(j)$ denote the index of the edge $\{v_j,w_j\}\subset T_j$ with the desirable property in the hypothesis of the proposition. Choose the input vector $\mathbf{t}$ whose $i$th component is given by $t_i=\begin{cases}
        3L, & i=n(j)\text{ for some }j\in\{1,\ldots,|\mathcal{E}_2|\} \\
        0, & \text{otherwise}
    \end{cases}$.
    Note that entries of the derivative matrix are either $0$ or of the form $-t+s+r$ where $t,s,r$ are side lengths of the triangle corresponding to that row. Thus, by our construction, the $n(j)$th column is a nonzero multiple of the $j$th standard basis vector, so the derivative matrix has full row rank.
\end{proof}

We now illustrate the flexibility of this result with several examples. The first example is the case of the \textbf{fish graph}. Here, the relevant mapping for pinned vertices $x_1, x_2$ is
$$ (A_{\Delta})^{\text{fish}}_{x_1,x_2}(E) = \lbrace (\text{Area}(x_1,x_2,x_3),\text{Area}(x_2,x_3,x_4),\text{Area}(x_4,x_5,x_6)) : x_3,x_4,x_5,x_6\in E \rbrace .$$

\begin{figure}[h]
    \centering
\begin{tikzpicture}[
  v/.style = {circle, fill=black, inner sep=2.2pt},
  mv/.style = {circle, draw=magenta!70!black, fill=magenta, inner sep=2.2pt}
]
  \coordinate (x3) at (0,0);
  \coordinate (x1) at (2,1.2);
  \coordinate (x2) at (2,-1.2);
  \coordinate (x6) at (-2,1.2);
  \coordinate (x4) at (-2,-1.2);
  \coordinate (x5) at (-3.4,0);

  \draw[line width=0.8pt] (x1)--(x2)--(x3)--(x1);     
  \draw[line width=0.8pt] (x5)--(x6)--(x4)--(x5);     
  \draw[line width=0.8pt] (x6)--(x3)--(x4);           

  \node[v, label=right:$v_4$] at (x3) {};
  \node[v, label=above:$v_5$] at (x1) {};
  \node[v, label=below:$v_6$] at (x2) {};
  \node[mv,label=above:$v_2$] at (x6) {};
  \node[v, label=below:$v_3$] at (x4) {};
  \node[mv,label=left:$v_1$] at (x5) {};
\end{tikzpicture}
\caption{The fish graph with pins at $v_1$ and $v_2$}
\end{figure}
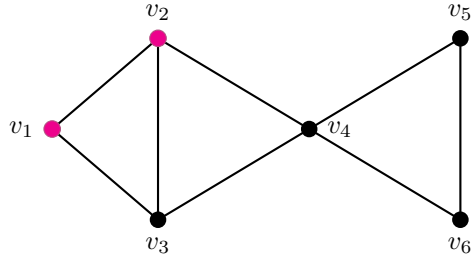

This relatively small example showcases triangles being attached along both edges and vertices, and our methods are able to deal with this without issue.

\begin{cor}\label{thm: fish}
Suppose $E\subset\mathbb{R}^d$, $d\geq 3$, is compact with $\dim(E)>(d+2)/2$. Then there exist points $x_1,x_2\in E$ such that
\[
\mathcal{L}^3((A_{\Delta})^{\text{fish}}_{x_1,x_2}(E))>0.
\]
\end{cor}

\begin{proof}
First, notice that after deleting the pinned edge $\{v_1,v_2\}$, this is a $2$-admissible pinned graph in the sense of \cite{BFOPRpaperI}, via the dismantling order going in descending numerical order. This gives a dimensional threshold of $(d+2)/2$ for getting an abundance of area vectors. Next, notice that we can choose $\{v_1,v_3\},\{v_3,v_4\},\{v_4,v_5\}$ as unshared unpinned edges. The claim follows immediately by Proposition \ref{prop: pinned suff transfer}.
\end{proof}

\begin{rem}
  In the fish graph example above, we can choose to pin a couple more vertices. For example, say we pin all vertices except $v_3$ and $v_4$, that is, $\mathcal{P}=\{v_1,v_2,v_5,v_6\}$. Since the edges $\{v_1,v_2\}$ and $\{v_5,v_6\}$ are fixed, we consider the pinned graph obtained by deleting those pinned edges. One can check that such a pinned graph is $3$-admissible (by deleting vertex $v_3$ of degree $3$ to get a new graph where $v_4$ also has degree $3$) and we can choose $\{v_1,v_3\},\{v_3,v_4\},\text{ and }\{v_4,v_5\}$ as exclusive unpinned unshared edges, so, as long as $\dim(E)>\frac{d+3}{2}$, there exists $x_1,x_2,x_5,x_6$ such that 
  $$\mathcal{L}^3(\{(\A(x_1,x_2,x_3),\A(x_2,x_3,x_4),\A(x_4,x_5,x_6))\colon x_3,x_4\in E\})>0.$$
\end{rem}

The second example is the case of a \textbf{wheel graph}, where $W_k$ denotes the graph join of $K_1$ (the graph with one vertex) and $C_k$ (the $k$-cycle graph). Here, the relevant set (for a wheel graph pinned at the distinguished center vertex and one external vertex) is
\[
(A_{\Delta})^{W_{k}}_{y,x_k}(E)=\{\text{Area}(x_1,x_2,y),\ldots,\text{Area}(x_{k-1},x_k,y),\text{Area}(x_{k},x_1,y):x_1,\ldots,x_{k-1}\in E\}.
\]
\begin{figure}[h]
    \centering
\begin{tikzpicture}[scale=1.6, every node/.append style={scale=1.2}]
\coordinate[vtx,magenta] (w) at (0,0) node[below] {$y$};
\foreach \i/\col in {1/magenta,2/,3/,4/,5/,6/} {
\coordinate[vtx,\col] (v\i) at (360*\i/6+120:1);
\draw (w) -- (v\i);
}
\draw (v1) -- (v2) -- (v3) -- (v4) -- (v5) -- (v6) -- (v1);

\draw (v1)--(v2);
\draw (v1)--(v6);
\draw (v1)--(w);

\node[left] at (v1) {$x_k$};
\node[left] at (v2) {$x_1$};
\node[left] at (v6) {$x_{k-1}$};

\end{tikzpicture}
\caption{The wheel graph $W_k$, with $k=6$ and pins in $y$ and $x_k$}
\end{figure}
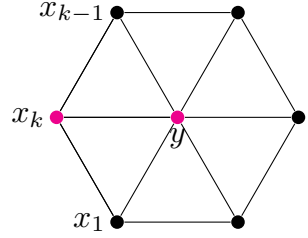
\begin{cor}\label{thm: wheel}
    Suppose $E\subset\mathbb{R}^d$ is compact with $\dim(E)>(d+3)/2$. Then for any $k\ge 3$, there exist points $x_k,y\in E$ such that
    \[
    \mathcal{L}^k((A_{\Delta})^{W_{k}}_{y,x_k}(E))>0
    \]
\end{cor}
\begin{proof}
   The pinned wheel graph is $3$-admissible (after deleting the pinned edge). For each triangle, its non-spoke edge is unpinned and unshared, so Proposition \ref{prop: pinned suff transfer} applies.
\end{proof}

One may ask about other graphs that result in cycles of triangles; one such example is the circumscribed flower graph given below:

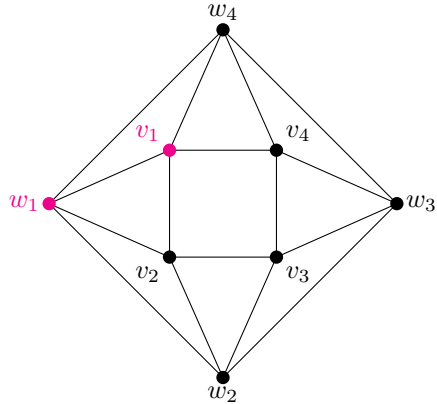
\begin{figure}[h]
    \centering
\begin{tikzpicture}[scale=1]

\foreach \i in {1,2,3,4} {
    \coordinate (v\i) at (90*\i+45:1);
    \coordinate (w\i) at (90*\i+90:2.3); 
}

\draw (v1) -- (w1) -- (v2) -- (w2) -- (v3) -- (w3) -- (v4) -- (w4) -- (v1);
\draw (v1) -- (v2) -- (v3) -- (v4) -- (v1);
\draw (w1) -- (w2) -- (w3) -- (w4) -- (w1);

\foreach \i in {2,3,4} {
    \fill[black] (v\i) circle (2.5pt);
    \fill[black] (w\i) circle (2.5pt);
}
\fill[magenta] (v1) circle (2.5pt);
\fill[magenta] (w1) circle (2.5pt);

\node[above left, color=magenta]  at (v1) {$v_1$};
\node[below left]  at (v2) {$v_2$};
\node[below right] at (v3) {$v_3$};
\node[above right] at (v4) {$v_4$};

\node[left, color=magenta]  at (w1) {$w_1$};
\node[below] at (w2) {$w_2$};
\node[right] at (w3) {$w_3$};
\node[above] at (w4) {$w_4$};

\end{tikzpicture}
\caption{The graph $CF_4$}
\end{figure}

More formally, the circumscribed flower graph with $k$ petals, call it $CF_k$, is given by starting with a $k$-cycle $C_k$, then introducing $k$ more vertices, each of which is connected to the vertices of a distinct edge in the original $k$-cycle, and then finally adding edges between the $k$ added vertices to form another $k$-cycle. When $k\ge 4$, this produces a planar graph with $2k$ triangles. Note that $CF_k$ is $4$-regular (and so $4$-admissible), which implies that $(d+4)/2$ should be the right threshold obtained from our methods. For simplicity, we will assume that we pin at the vertices $w_1$ and $v_1$, and we will label the vertices in the outer $k$-cycle with $w$'s (or $y$'s for their realizations in Euclidean space) and the vertices in the inner $k$-cycle with $v$'s (or $x$'s for their realizations in Euclidean space).
\begin{align*}
(A_{\Delta})_{x_1,y_1}^{CF_k}(E)&=\{\big(\text{Area}(x_1,x_2,y_1),\text{Area}(x_2,x_3,y_2),\ldots,\text{Area}(x_k,x_1,y_k),\\&\text{Area}(y_1,y_2,x_2),\text{Area}(y_2,y_3,x_3),\ldots,\text{Area}(y_k,y_1,x_1)\big):x_2,\ldots,x_k,y_2,\ldots,y_k\in E\}
\end{align*}
\begin{cor}
    Let $E\subset \mathbb{R}^d$ be a compact subset with $\dim(E)>(d+4)/2$. Then there exist $y_1,x_1\in E$ such that
    \[
    \mathcal{L}^{2k}((A_{\Delta})_{x_1,y_1}^{CF_k}(E))>0.
    \]
\end{cor}
\begin{proof}
    The graph is $4$-regular, hence $4$-admissible. Selecting edges of the form $(v_i,v_{i+1})$ and $(w_i,w_{i+1})$ for each $i$ (where we interpret $k+1=1$) gives an unshared unpinned edge for each triangle, and then the claim follows by Proposition \ref{prop: pinned suff transfer}.
\end{proof}

While the conditions of Propositions \ref{prop suff transfer} and \ref{prop: pinned suff transfer} are sufficient for getting a dimensional threshold for many examples, they are known not to be necessary. Consider the following graph $P$ pictured below:
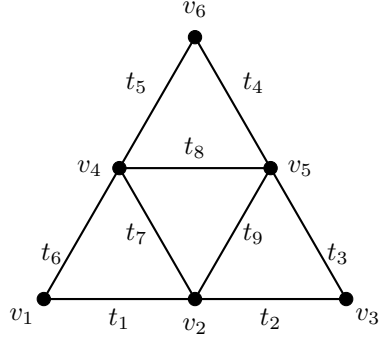
\begin{figure}[h]
    \centering
\begin{tikzpicture}
  \coordinate (v1) at (0,0);
  \coordinate (v2) at (2,0);
  \coordinate (v3) at (4,0);
  
  \coordinate (v4) at (1, {sqrt(3)});
  \coordinate (v5) at (3, {sqrt(3)});
  
  \coordinate (v6) at (2, {2*sqrt(3)});

  \draw[thick] (v1) -- node[below] {$t_1$} (v2);
  \draw[thick] (v2) -- node[below] {$t_2$} (v3);
  \draw[thick] (v3) -- node[below right] {$\,\,t_3$} (v5);
  \draw[thick] (v5) -- node[above right] {$t_4$} (v6);
  \draw[thick] (v6) -- node[above left] {$t_5$} (v4);
  \draw[thick] (v4) -- node[below left] {$t_6\,\,$} (v1);
  
  \draw[thick] (v2) -- node[left] {$t_7$} (v4);
  \draw[thick] (v4) -- node[above] {$t_8$} (v5);
  \draw[thick] (v5) -- node[right] {$t_9$} (v2);

  \filldraw (v1) circle (2.5pt) node[below left] {$v_1$};
  \filldraw (v2) circle (2.5pt) node[below=4pt] {$v_2$};
  \filldraw (v3) circle (2.5pt) node[below right] {$v_3$};
  \filldraw (v4) circle (2.5pt) node[left=3pt] {$v_4$};
  \filldraw (v5) circle (2.5pt) node[right=3pt] {$v_5$};
  \filldraw (v6) circle (2.5pt) node[above=4pt] {$v_6$};
\end{tikzpicture}
\caption{The graph $P$; its central triangle has no unshared edges, so Proposition \ref{prop: pinned suff transfer} does not apply.} \label{fig:pyramidof4triangles}
\end{figure}

If we enumerate the triangles in the order $(t_1,t_6,t_7),(t_2,t_3,t_9),(t_4,t_5,t_8),(t_7,t_8,t_9)$, then we can readily compute the derivative matrix of the length-to-area map to be:
\[
D\Phi(t)=\scalebox{0.65}{$\begin{bmatrix}
    -t_1+t_6+t_7 & 0 & 0 & 0 & 0 & t_1-t_6+t_7 & t_1+t_6-t_7 & 0 & 0 \\
    0 & -t_2+t_3+t_9 & t_2-t_3+t_9 & 0 & 0 & 0 & 0 & 0 & t_2+t_3-t_9 \\
    0 & 0 & 0 & -t_4+t_5+t_8 & t_4-t_5+t_8 & 0 & 0 & t_4+t_5-t_8 & 0 \\
    0 & 0 & 0 & 0 & 0 & 0 &-t_7+t_8+t_9 & t_7-t_8+t_9 & t_7+t_8-t_9
 \end{bmatrix}$}
\]

This can easily be seen to be full rank for a generic input; one approach to get a specific set of values that works is to consider columns $1,2,4,7$ and set $t_1=2,t_2=1,t_4=1,t_7=1$ and all other variables to be $0$. Indeed, these variables correspond to choosing one edge from each of the triangles in the graph.

In fact, there is a conjectured equivalent characterization of edge area transferal condition due to \cite{LNPR} (see Conjecture 12). We'll introduce some terminology to explain the condition and formulate it in the language of our paper.

\begin{defn}\label{def:triangleedgechoice}
    Given a simplicial complex $G=(\mathcal{V},\mathcal{E})$, we say that a map $f:\mathcal{E}_2\rightarrow \mathcal{E}_1$ is a \textbf{triangle edge choice function} if $f(T)\subset T$ for every $T\in\mathcal{E}_2$.
\end{defn}

\begin{conj}
    Let $G=(\mathcal{V},\mathcal{E})$ be a simplicial complex. Then length-area transferal  holds for $G$ if and only if there exists an injective triangle edge choice function.
\end{conj}

In fact, this can be seen to be necessary by Hall's Marriage Theorem (see \cite{HallSource}, for instance). That theorem says in our setting that the existence of an injective triangle edge choice function is equivalent to the property that there does not exist any subcollection of triangles in the simplicial complex with the property that the subcomplex induced by those triangles has more triangles than edges. Hence, if there is not an injective triangle edge choice function, then we can find a subcomplex of $G$, call it $H$, with the property that $H$ has more triangles than edges. Restricting the derivative matrix to only the rows corresponding to triangles in $H$, we see that this submatrix has more rows than columns, implying the rows are linearly dependent. Thus, the entire derivative matrix has linearly dependent rows and cannot be full rank. We refer the reader to Example 11 in \cite{LNPR} for an illuminating example relating to this.

In the unpinned setting, we are able to get closer to the conjecture by proving Theorem \ref{thm: mainJacobianthm} which we restate below. Observe, in particular, that the statement below can take care of the example in Figure \ref{fig:pyramidof4triangles}, which was not contemplated by the Proposition \ref{prop suff transfer}.

\begin{prop}\label{restatemainthmtoprove}
    Let $G=(\mathcal{V},\mathcal{E})$ be a $2$-dimensional simplicial complex. If there exists an injective triangle edge choice function $f$ with the property that no triangle has all its edges chosen by $f$, then length-area transferal holds for $G$.
\end{prop}
\begin{proof}
    The basic idea is to consider the square submatrix obtained by selecting only the columns of $D\Phi$ corresponding to the edges that are chosen by $f$. The condition that no triangle has all its edges chosen will then imply that each row of this submatrix has at most $2$ nonzero entries, and choosing the selected input variables in a clever way will allow us to guarantee that we can get the matrix $D\Phi$ into a triangular form with nonzero diagonal entries.

    More precisely, we will describe an algorithm to order the triangles and edges that will put $D\Phi$ in a simple form. Suppose we have already chosen triangles $F_1,\ldots, F_k$ and let $t_1,\ldots, t_k$ correspond to the edges $f(F_1),\ldots, f(F_k)$, respectively. If $F_k$ has a second edge that was selected by $f$ and this edge is not equal to any of $f(F_1),\ldots, f(F_k)$, then select $F_{k+1}$ to be the triangle such that $f(F_{k+1})$ corresponds to that edge. Otherwise, then choose $F_{k+1}$ to be any unselected triangle. In either case, let $t_{k+1}$ correspond to the edge $f(F_{k+1})$. Let $l$ denote the total number of triangles, and consider the first $l$ columns of $D\Phi$. We claim that for a certain choice of inputs, the first $l$ columns of $D\Phi(t)$ form an invertible matrix, which we will call $M$.

    By construction, we have that the $(i,j)$th entry of $M$ is $0$ whenever $j>i+1$. We can relabel the selected variables $t_i$ as $t_i^B$ to keep track of which block they were selected in, and then recognize that $M$ is a block lower triangular matrix whose diagonal block structure corresponds to cycles of the form
    \[
    B=\begin{bmatrix}
        -t_1^B+t_2^B+u_1^B & t_1^B-t_2^B+u_1^B & 0 & 0 & 0\\
        0 & -t_2^B+t_3^B+u_2^B & t_2^B-t_3^B+u_2^B & 0  & 0\\
        \vdots & \ddots & \ddots & \ddots & \vdots \\
        0 & 0 & \ldots & -t_{r-1}^B+t_r^B+u_{r-1}^B & t_{r-1}^B-t_r^B+u_{r-1}^B \\
        * & * & \ldots & * & -t_r^B+?+u_r^B
    \end{bmatrix}
    \]
    where the bottom row could either have the case where each $*$ entry is $0$ and the last entry is $-t_r^B+u_r^B+(u_r')^B$, or there is a block $B'$ that may or not be the same as $B$ such that some $*$ entry takes the form $-t_i^{B'}+t_r^B+u_r^B$ and the last entry is $t_i^{B'}-t_r^B+u_r^B$ for $t_i^{B'}\neq t_{r-1}^B$ (notice that $t_{i}^{B'}=t_{r-1}^{B}$ is impossible since it would imply that the triangles corresponding to the $(r-1)$st and $r$th row of $B$ have two common edges in $G$, but every two distinct triangles share at most one edge). Here, we are using the notation $u_j^B$ to denote the unselected edge of the $j$th triangle in the block $B$, and $(u_r')^{B}$ in the event that the triangle that selects $t_r^B$ has two unselected edges. Note that while all the $t_i^B$ are distinct, that is not the case for the variables $u_i^B$,  and it could be that $u_i^B=u_{j}^{B'}$ for $(i,B)\neq (j,B')$. 
    
    $M$ is invertible if and only if each of its diagonal blocks is invertible, so it suffices to make all the blocks $B$ invertible. We first put all the blocks $B$ in triangular form. In the first case above, where the last row has only one nonzero entry $-t_r^B+u_r^B+(u_r')^B$, $B$ is already upper triangular, so we leave it as it is. In the second case, in the row corresponding to the triangle $F_r^B$ such that $f(F_r^B)=t_r^B$, there is another nonzero entry $-t_i^{B'}+t_r^B+u_r^B$. The variable $t_i^{B'}$ could, in principle, have been chosen by a different block $B'\neq B$. If that is the case, the last row of block $B$ has only one non-zero entry, and $ B$ is already upper triangular. In the case that $t_i^{B'}=t_i^{B}$, that is, it was chosen by block $B$, then $1\leq i<r-1$ we make $B$ upper triangular by forcing the entry $-t_i^B+t_r^B+u_r^B$ to be equal to $0$, by breaking the algebraic independence and by setting $t_r^B=t_i^B-u_r^B$. 

After carrying out this triangularization for every diagonal block, we obtain
a collection of relations of the form
\[
t_r^B=t_i^B-u_r^B
\]
for some of the selected variables. Let \(\mathcal F\) be the set of all
underlying edge variables which remain free after imposing these relations.
Here we list each underlying edge only once. Thus \(\mathcal F\) contains all
selected variables \(t_j^B\) except the variables \(t_r^B\) which have been
defined by one of the above relations, together with all unselected edge
variables which occur in the blocks, again listed without repetition.

With these choices, each diagonal block is upper triangular and every diagonal entry is nonzero. Indeed, each such diagonal entry is a
linear polynomial with integer coefficients in the variables from
\(\mathcal F\). Moreover, this linear polynomial is not identically zero. For
the rows not affected by the triangularization relation, this follows from
the fact that the three edges of a triangle are distinct, and from the
injectivity of the choice function. For a row affected by the relation
\(t_r^B=t_i^B-u_r^B\), the last diagonal entry becomes
\[
2u_r^B,
\]
while the preceding diagonal entry involving \(t_r^B\) becomes
\[
t_i^B-t_{r-1}^B+u_{r-1}^B-u_r^B.
\]
This is again a nonzero linear polynomial in the free variables: the edges
\(t_i^B\) and \(t_{r-1}^B\) are distinct selected edges, and the edges
\(u_{r-1}^B\) and \(u_r^B\) are distinct unselected edges, since otherwise the
two corresponding triangles would share two edges.

After all substitutions have
been made, each diagonal entry is a linear nonzero polynomial with integer coefficients
in the variables from \(\mathcal F\), and we can now choose values for those variables that are algebraically
independent over \(\mathbb Q\), so no nonzero polynomial with rational
coefficients in these variables can vanish. In particular, none of the
nonzero linear polynomials appearing on the diagonal vanishes. Hence, every
diagonal block is invertible. Since the full selected minor is block lower
triangular with these diagonal blocks, the selected minor has a nonzero
determinant. Therefore \(D\Phi\) has full row rank for this choice of input
variables, and length-area transferal holds for \(G\).

\end{proof}

We will now present a few more ad hoc observations that may be of interest while the general equivalent condition is out of reach. In the pinned setting, recall that we allow ourselves to pin both vertices of an edge, so long as we never pin all three vertices of a triangle whose area is under consideration. In such a setting, we think of any side length whose vertices are both pinned to be ``frozen" to some fixed positive value which we cannot control. One possible cause for concern is whether we could get extremely unlucky and freeze side lengths in such a way that the relevant Jacobian is never full rank, even if for generic choices of frozen side lengths, the Jacobian is full rank. In the setting of Proposition \ref{prop: pinned suff transfer}, this was never an issue due to being able to choose the unpinned sides to be sufficiently large. We illustrate another approach with an example that is not addressed by Proposition \ref{prop: pinned suff transfer}.

Consider the complete graph $K_4$ with three of its vertices pinned so that the relevant pinned area set becomes 
$$(A_{\Delta})^G_{x_1,x_2,x_3}(E)=\{\A(x_1,x_3,x_4),\A(x_2,x_3,x_4),\A(x_1,x_2,x_4)\colon x_4\in E\}.$$
We want to know if we can transfer positive length results from the three unpinned edges to the three unpinned triangle areas in $K_4$.

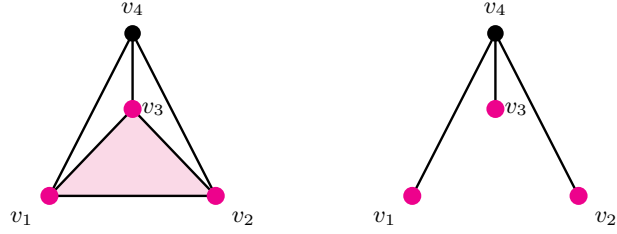
\begin{figure}[ht]
\centering
\begin{tikzpicture}[
    vertex/.style={circle, fill=black, inner sep=2.2pt},
    pin/.style={circle, fill=magenta, inner sep=2.4pt},
    edge/.style={line width=0.9pt},
    lab/.style={font=\small}
]

\coordinate (v1) at (0,0);
\coordinate (v2) at (2.2,0);
\coordinate (v3) at (1.1,1.15);
\coordinate (v4) at (1.1,2.15);

\fill[magenta!18] (v1) -- (v2) -- (v3) -- cycle;

\draw[edge] (v1)--(v2)--(v3)--(v1);
\draw[edge] (v4)--(v1);
\draw[edge] (v4)--(v2);
\draw[edge] (v4)--(v3);

\node[pin,label={[lab]below left:$v_1$}] at (v1) {};
\node[pin,label={[lab]below right:$v_2$}] at (v2) {};

\node[pin] at (v3) {};
\node[lab, anchor=west] at ($(v3)$) {$v_3$};
\node[vertex,label={[lab]above:$v_4$}] at (v4) {};


\begin{scope}[xshift=4.8cm]
    \coordinate (v1s) at (0,0);
    \coordinate (v2s) at (2.2,0);
    \coordinate (v3s) at (1.1,1.15);
    \coordinate (v4s) at (1.1,2.15);

    \draw[edge] (v4s)--(v1s);
    \draw[edge] (v4s)--(v2s);
    \draw[edge] (v4s)--(v3s);

    \node[pin,label={[lab]below left:$v_1$}] at (v1s) {};
    \node[pin,label={[lab]below right:$v_2$}] at (v2s) {};
\node[pin] at (v3s) {};
\node[lab, anchor=west] at ($(v3s)$) {$v_3$};
    \node[vertex,label={[lab]above:$v_4$}] at (v4s) {};

\end{scope}

\end{tikzpicture}

\caption{$K_4$ with $3$ of its vertices pinned, and the associated $3$-pinned star obtained by deleting the frozen edges and the fully pinned triangle ($2$-edge in magenta). }
\label{fig:pinned-k4-star}
\end{figure}
Using $t_1,t_2,t_3$ to denote the unfrozen edge lengths and $s_1,s_2,s_3$ to denote the frozen side lengths, we compute the relevant Jacobian to be
\[
\begin{bmatrix}
    -t_1+t_2+s_1 & t_1-t_2+s_1 & 0 \\
    0 & -t_2+t_3+s_2 & t_2-t_3+s_2 \\
    -t_1+t_3+s_3 & 0 & t_1-t_3+s_3
\end{bmatrix}
\]
Given $s_1,s_2,s_3>0$ we want to show that we can always find $t_1,t_2,t_3$ making this matrix full rank. In this case, this is easy to do; choose $t_3=t_1-s_3$ and it becomes upper triangular; then for generic choices of $t_1,t_2$, we see that all the diagonal entries are nonzero. In fact, this argument holds more generally for any wheel graph whose outer $k$-cycle of vertices are all pinned.

\begin{rem}\label{rem:remarkaboutcompletegraph}
For the unpinned complete graph \(K_n\) in $n$ vertices, there is an immediate rank obstruction
when \(n\geq 6\). Indeed, the number of edge variables is
$\binom{n}{2}$,
whereas the number of triangles is
$\binom{n}{3}$.
For \(n\geq 6\), one has $\binom{n}{2}<\binom{n}{3}.$ Thus the Jacobian of the length-to-area map cannot have full row rank, and length-area transferal cannot hold for \(K_n\).

The cases \(K_3\) and \(K_4\) were discussed above; in particular, edge area
transferal holds even after pinning all but one vertex. The remaining
borderline case is \(K_5\), for which
\[
\#\mathcal E_1=\binom{5}{2}=\binom{5}{3}=\#\mathcal E_2.
\]
In this case the Jacobian of the map
\[
\Phi:(t_{ij})_{(v_i,v_j)\in \mathcal E_1}
\mapsto
\bigl(\alpha(t_{ij},t_{ik},t_{jk})\bigr)_{(v_i,v_j,v_k)\in \mathcal E_2}
\]
is a \(10\times 10\) matrix. Evaluating at the point \(t_{ij}=1\) for every
edge \((v_i,v_j)\in \mathcal E_1\), and using the normalization of
\(\alpha\) from Section~\ref{sec: jacobianmethod}, this Jacobian becomes the
triangle-edge incidence matrix of \(K_5\). In particular, each row has exactly
three entries equal to \(1\) and all other entries equal to \(0\). A direct
computation shows that its determinant is nonzero; more precisely, it is
\(\pm 48\), depending on the orderings chosen for the edges and triangles.
Therefore edge area transferal holds for \(K_5\).
\end{rem}

Since length-area transferal holds for $K_{k+1}$ for $2\leq k\leq 4$, and $K_{k+1}$ is a $k$-dimensional simplex, one can check what thresholds are available for positive measure of $\Delta^{K_{k+1}}(E)$. By plugging Du and Zhang's spherical average decay estimate \cite[Theorem 2.8]{DZ2019} into group action machinery in \cite{GILP15}, one gets the threshold 
$$E\subset \R^d,\,d\geq 3,\,\dim(E)>\alpha_{\textnormal{DZ/GILP}}(d,k):=\frac{kd^2}{(k+1)d-1}\Rightarrow \mathcal{L}^{{k+1}\choose 2}(\Delta^{K_{k+1}}(E))>0.$$
Let us focus on $k=d$ (complete graph on $(d+1)$ vertices in $\R^d$). Then one gets
$$E\subset \R^3,\,\dim(E)>\alpha_{\textnormal{DZ/GILP}}(3,3):=\frac{27}{11}\Rightarrow \mathcal{L}^{6}(\Delta^{K_4}(E))>0$$
and 
$$E\subset \R^4,\,\dim(E)>\alpha_{\textnormal{DZ/GILP}}(4,4):=\frac{64}{19}\Rightarrow \mathcal{L}^{10}(\Delta^{K_5}(E))>0.$$
By length-area transferal, we see that the set of triangle area vectors showing up in tetrahedra with vertices in $E\subset \R^3,\, \dim(E)>\frac{27}{11}$, has positive measure in $\R^4$. Similarly $(A_{\Delta})^{K_5}(E)\subset \R^{10}$ has positive Lebesgue measure if $E\subset \R^4$ with $\dim(E)>\frac{64}{19}$. One can also get a positive measure for areas of triangles with vertices in $E\subset \R^2$ for $\dim(E)>8/5$, but that is not as good as Yavicoli and Shmerkin's threshold $1$.

\begin{prop}\label{prop:banana-transferal}
Let \(G\) be the banana graph of Figure \ref{fig:banana} with vertices \(v_1,\ldots,v_5\) and with
\(2\)-simplices
\[
\begin{array}{lll}
T_1=\{v_3,v_4,v_5\},&
T_2=\{v_1,v_3,v_4\},&
T_3=\{v_1,v_4,v_5\},\\
T_4=\{v_1,v_3,v_5\},&
T_5=\{v_2,v_3,v_4\},&
T_6=\{v_2,v_4,v_5\},\\
T_7=\{v_2,v_3,v_5\}.&
\end{array}
\]
Then edge area transferal holds for \(G\).
\end{prop}

\begin{proof}
We label the nine side lengths as follows:
\[
\begin{array}{lll}
t_1=t_{34},&
t_2=t_{45},&
t_3=t_{53},\\
t_4=t_{13},&
t_5=t_{14},&
t_6=t_{15},\\
t_7=t_{23},&
t_8=t_{24},&
t_9=t_{25}.
\end{array}
\]
Thus \(t_1,t_2,t_3\) are the side lengths of the central triangle,
\(t_4,t_5,t_6\) are the side lengths from \(v_1\) to the central triangle, and
\(t_7,t_8,t_9\) are the side lengths from \(v_2\) to the central triangle.

We use the polynomial from Section~\ref{sec: jacobianmethod},
\[
\alpha(x,y,z)
=
-\frac{x^2}{2}-\frac{y^2}{2}-\frac{z^2}{2}
+xy+xz+yz.
\]
Then
\[
\nabla \alpha(x,y,z)
=
(-x+y+z,\ x-y+z,\ x+y-z),
\]
and hence
\[
\nabla \alpha(1,1,1)=(1,1,1).
\]

With the above labeling, the length-to-area map has seven coordinates
\[
\begin{aligned}
A_1=&\alpha(t_1,t_2,t_3),
A_2=\alpha(t_1,t_4,t_5),A_3=\alpha(t_2,t_5,t_6),\\
A_4=\alpha(t_3,t_4&,t_6),
A_5=\alpha(t_1,t_7,t_8),
A_6=\alpha(t_2,t_8,t_9),
A_7=\alpha(t_3,t_7,t_9).
\end{aligned}
\]
We evaluate the Jacobian at the point
\[
t_1=t_2=t_3=t_4=t_5=t_6=t_7=t_8=t_9=1.
\]
Since \(\nabla\alpha(1,1,1)=(1,1,1)\), every nonzero entry of this Jacobian
is equal to \(1\).

Now restrict the Jacobian to the seven columns corresponding to
\[
t_1,\ t_2,\ t_3,\ t_4,\ t_5,\ t_6,\ t_7.
\]
With rows ordered as \(A_1,A_2,\ldots,A_7\), the corresponding \(7\times 7\)
minor is
\[
M=
\begin{pmatrix}
1&1&1&0&0&0&0\\
1&0&0&1&1&0&0\\
0&1&0&0&1&1&0\\
0&0&1&1&0&1&0\\
1&0&0&0&0&0&1\\
0&1&0&0&0&0&0\\
0&0&1&0&0&0&1
\end{pmatrix}.
\]
A direct computation gives
\[
\det M=4.
\]
Therefore this minor is nonzero. Hence the Jacobian of the length-to-area map
has rank \(7\), which is the number of area coordinates. Thus the map has full
row rank at this input, and edge area transferal holds for the banana graph.

\end{proof}

\subsection{Comments on higher-dimensional volumes}
Though our discussion so far has been focused on finding positive measure sets of area vectors, it is not so difficult to generalize some of our methods to higher dimensional $n$-volume vectors. We do not attempt to pursue this avenue in any great generality or detail here, instead contenting ourselves with sketching a few simple ideas. As before, Theorem \ref{thm:k starinIPPS} will be useful. The substitute for Heron's formula is the fact that the square of the $n$-volume of an $n$-simplex can be computed using \eqref{eq Cayley Menger}, which says that the square of the volume is given by a polynomial function of the lengths of the edges. Together, these observations can be used to get results for ``fans of tetrahedra", for instance, such as the simplicial complex pictured below this consisting of two tetrahedra sharing a face, a graph which we will call $T_2$.

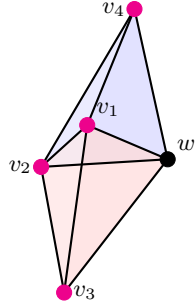
\begin{figure}[h]
    \centering
\begin{tikzpicture}[scale=1.7, tdplot_main_coords,
  vertex/.style={circle,fill=black,draw=black,inner sep=2pt},
  edge/.style={line width=0.8pt},
  hidden/.style={edge,dashed,opacity=0.5},
  lab/.style={font=\small}
]

\coordinate (v1) at (0.5,0.2,0.2);
\coordinate (v2) at (1,0,0);
\coordinate (w) at (0.5,0.866,0);  

\coordinate (v4) at (-0.3,0.3,0.9);   
\coordinate (v3) at (1.3,0.3,-0.9);   

\fill[gray!15] (v1)--(v2)--(w)--cycle;

\fill[blue!20,opacity=0.7] (v4)--(v1)--(w)--cycle;
\fill[blue!10,opacity=0.7] (v4)--(v2)--(w)--cycle;
\fill[blue!8,opacity=0.7] (v4)--(v1)--(v2)--cycle;

\fill[red!15,opacity=0.7] (v3)--(v1)--(w)--cycle;
\fill[red!10,opacity=0.7] (v3)--(v2)--(w)--cycle;
\fill[red!8,opacity=0.7] (v3)--(v1)--(v2)--cycle;

\draw[edge,thick] (v1)--(v2)--(w)--cycle;

\draw[edge] (v4)--(v1);
\draw[edge] (v4)--(v2);
\draw[edge] (v4)--(w);

\draw[edge] (v3)--(v1);
\draw[edge] (v3)--(v2);
\draw[edge] (v3)--(w);


\node[vertex,magenta] at (v1) {};
\node[vertex, magenta] at (v2) {};
\node[vertex] at (w) {};
\node[vertex, magenta] at (v4) {};
\node[vertex, magenta] at (v3) {};

\node[lab,above right]  at (v1) {$v_1$};
\node[lab,left] at (v2) {$v_2$};
\node[lab,above right]       at (w) {$w$};
\node[lab,left]        at (v4) {$v_4$};
\node[lab,right]       at (v3) {$v_3$};

\end{tikzpicture}
\caption{The graph $T_2$ with the two tetrahedra of interest shaded}
\end{figure}

Let
\[
(V_{3-\Delta})_{x_1,x_2,x_3,x_4}^{T_2}(E):=\{(\text{Vol}(x_1,x_2,x_3,y),\text{Vol}(x_1,x_2,x_4,y)):y\in E\}
\]
denote the collection of volume vectors associated to the graph $T_2$ that are realized by point configurations in a compact set $E\subset \mathbb{R}^d$. The following is a sample result along the lines of what we did for areas.
\begin{thm}
    Let $E\subset\mathbb{R}^d$ be a compact subset with $\dim(E)>(d+4)/2$. Then there exist points $x_1,x_2,x_3,x_4\in E$ such that
    \[
    \mathcal{L}^2((V_{3-\Delta})_{x_1,x_2,x_3,x_4}^{T_2}(E))>0.
    \]
\end{thm}
We do not give a detailed proof here as the argument is analogous to the proof of Theorem \ref{thm: k pinned fans}, instead giving merely a sketch for the reader. The threshold is now $(d+4)/2$ because we need to use a version of Theorem \ref{thm:k starinIPPS} for 4-stars instead of 3-stars. If we let $t_i=|x_i-y|$, then the analog of the mapping $\Phi$ we would define would be given by
\[
\Phi(t_1,t_2,t_3,t_4)=(t_1,t_2,\text{Vol}(x_1,x_2,x_3,y)^2,\text{Vol}(x_1,x_2,x_4,y)^2)
\]
where the squares of the volumes can be written as polynomials in terms of the distances between the various points realized in $E$. The derivative matrix $D\Phi$ is triangular and vanishes on some variety which has Lebesgue measure 0, so after excising a small portion of $E$, the Jacobian is bounded above and below. Repeating the argument from Theorem \ref{thm: k pinned fans} then gives the claim.

Other generalizations involving larger fans of tetrahedra, or even involving $n$-volumes of fans of $n$-simplices are possible using similar methods, but we do not pursue those here, instead leaving them for the interested reader.

\subsection{Nonempty interior transferal}
In addition to transferring information about distance sets having positive measure to area sets, we can also transfer the property of having nonempty interior, provided that the length-to-area map has full row rank somewhere. The basic idea is as follows: let $U$ denote an open subset of the set of length vectors and let $\Phi$ denote the length-to-area map. Then $D\Phi$, being a polynomial map, is full rank except on a lower dimensional (closed) subvariety $V$. Consider $\Phi|_{U\setminus V}$; on this domain, $\Phi$ is a submersion and is thus an open map, so its image is also open. We can use this to deduce some corollaries.

For example, length-area transferal holds for unpinned triangles since the length-to-area map for a triangle has full
rank away from a proper algebraic subvariety. By a result of Palsson and Romero Acosta \cite{PRA23}, for $d\geq 4$, as long as $E\subset \R^d$ with $\dim(E)>\frac{2d+3}{3}$, one has $\text{int}(\Delta^{K_3}(E))\neq \emptyset$ where $\Delta^{K_3}(E)=\{(|x_1-x_2|,|x_2-x_3|,|x_3-x_1|)\colon x_1,x_2,x_3\in E\}$. The observation above implies that the triangle area set \[ A_{\Delta}(E) = \{ \A(x_1,x_2,x_3):x_1,x_2,x_3\in E \} \] has nonempty interior under the same dimensional assumption. Further extensions to simplices can be achieved by using the more recent paper by the same authors \cite{PRA25}.

One can also transfer pinned nonempty interior results for \(k\)-stars. As an
illustration, consider the wheel graph \(W_6\), with all vertices pinned except
for the center vertex \(y\). Label the pinned vertices by
\(v_1,\ldots,v_6\), cyclically, and consider the six triangles
\[
\{v_i,v_{i+1},y\},
\qquad
1\leq i\leq 6,
\]
where \(v_7=v_1\). The outer cycle is fully pinned. Let
\(s_i\) denote the frozen length of the pinned edge \(\{v_i,v_{i+1}\}\), and
let \(t_i\) denote the length of the edge \(\{v_i,y\}\). The length-to-area map
is a map from \(\mathbb R^6\) to \(\mathbb R^6\), and its Jacobian has the form
\[
\begin{pmatrix}
-t_1+t_2+s_1&t_1-t_2+s_1&0&0&0&0\\
0&-t_2+t_3+s_2&t_2-t_3+s_2&0&0&0\\
0&0&-t_3+t_4+s_3&t_3-t_4+s_3&0&0\\
0&0&0&-t_4+t_5+s_4&t_4-t_5+s_4&0\\
0&0&0&0&-t_5+t_6+s_5&t_5-t_6+s_5\\
-t_1+t_6+s_6&0&0&0&0&t_1-t_6+s_6
\end{pmatrix}.
\]
For fixed \(s_1,\ldots,s_6>0\), choose \(t_1,\ldots,t_5\) generically so that
the first five diagonal entries are nonzero. Then set
\[
t_6=t_1-s_6.
\]
With this choice, the lower-left entry vanishes, while the last diagonal entry
becomes
\[
t_1-t_6+s_6=2s_6\neq 0.
\]
Thus the Jacobian has full rank for some choice of \(t_1,\ldots,t_6\). It
follows that, whenever the pinned \(6\)-star length set
\[
\Delta^{6\operatorname{-star}}_{x_1,\ldots,x_6}(E)
\]
has nonempty interior, the corresponding pinned wheel area set also has
nonempty interior.

In particular, the forthcoming work~\cite{BOP2026} gives pinned nonempty
interior results for \(k\)-stars. More precisely, for
\(1\leq k<d/2\), if
\[
\dim(E)>
\alpha^{\circ}(d,k)
:=
\frac{d+2k-1}{2}
+\frac{1}{4}
+\frac{4k+1}{4(2d+1)},
\]
then there exist pins \(x_1,\ldots,x_k\in E\) such that
\[
\operatorname{int}
\bigl(
\Delta^{k\operatorname{-star}}_{x_1,\ldots,x_k}(E)
\bigr)
\neq \emptyset.
\]
Combining this with the preceding Jacobian argument yields a nonempty interior
result for areas of the $6$ relevant triangles in $W_6$ as long as $\dim(E)>\alpha^{\circ}(d,6)$, and similar conclusion holds for general $k\geq 3$ and pinned wheel $W_k$ with pins in all but the central vertex.

\section{Areas of triangles in the plane via projections}\label{sec: projectionmethod}

The goal of this section is to develop projection-theoretic building blocks for triangle area configurations in the plane. These arguments replace the Jacobian method used in higher dimensions.

Our first result is a strengthened form of the theorem of Shmerkin and Yavicoli. For planar sets of Hausdorff dimension strictly larger than $1$, we show not only that some pinned triangle area set has positive Lebesgue measure, but also that the corresponding good pairs of pins occur abundantly, in the sense that they form a set of positive product measure.

This strengthened form of the Shmerkin--Yavicoli theorem allows us to use double-pinned statements for triangle areas as building blocks for more complicated graphs assembled from triangles. In our abundance statement, a good pair of pins is required to satisfy a stronger $L^2$ projection condition. This formulation is better suited to Fubini-type gluing arguments and is crucial for handling area configurations arising from edge-attached triangle chains in the plane. The details of this transfer to edge-attached triangle chains are given in Proposition~\ref{prop:trianglesbyedges} in the appendix. Moreover, the abundance of choices for the first pin allows us to obtain positive Lebesgue measure for area vectors associated with vertex-attached triangle chains, through the vertex-gluing mechanism proved in Proposition~\ref{prop: trianglesbyvertices} in the appendix. These ingredients yield Corollaries \ref{cor: vertextrianglechain} and \ref{cor:edgetrianglechains} for vertex-attached and edge-attached chains of triangles in the plane.

The planar fish graph is treated by combining the ideas used for the two types of triangle chains described above. That verifies the threshold \(\dim(E)>1\) in $d=2$, stated in Theorem \ref{thm: fishAllD}.

 Let us introduce some notation that will be useful for the rest of the paper. The radial projection with center $x_1 \in \R^2$, is given as
\[
\pi_{x_1} : \R^2\setminus\{x_1\} \to \mathbb{S}^1, 
\quad y \mapsto \frac{y - x_1}{|y - x_1|}.
\]

We recall the following key estimate proved by Orponen \cite{Orponen19} (also see \cite[Proposition 3.11]{KeletiShmerkin} for this particular statement). 

\begin{thm}[Orponen's radial projection estimate] \label{thm:OrponenL^pprojection} For every $\alpha>1$, there exists $p=p(\alpha)>1$ such that the following holds. Let $\mu_1$ and $\mu_2$ be compactly supported $\alpha$-Frostman measures with disjoint supports in $\R^2$. Then 
$$\int \|(\pi_{x_1})_{*}(\mu_2)\|_{L^p(\mathbb{S}^1)}^pd\mu_1(x_1)<\infty.$$
    
\end{thm}

Let $\rho_{\theta,x_1}$ be the orthogonal projection onto the affine line of direction $\theta\in [0,2\pi]\simeq \mathbb{S}^1$ passing through $x_1$, i.e.
\[
\rho_{\theta,x_1}(y) : \R^2 \to \R,
\quad y \mapsto (y - x_1)\cdot (\cos\theta, \sin\theta).
\]

Observe that
\[
(y-x_1)\cdot (\cos\theta,\sin\theta) 
= y\cdot (\cos\theta,\sin\theta) - x_1 \cdot (\cos\theta,\sin\theta),
\]
and since the Lebesgue measure is translation invariant 
\[
\mathcal{H}^1(\rho_{\theta,x_1}(E_3)) > 0
\quad \iff \quad
\mathcal{H}^1(\rho_{\theta,0}(E_3)) > 0.
\]

Denote $\rho_{\theta,0}$ simply $\rho_{\theta}$.

\subsection{A projection-theoretic building block at threshold $1$.}

In this subsection, we prove a strengthened version of Shmerkin and Yavicoli's result that guarantees an abundance of good pairs of pins satisfying an $L^2$ orthogonal projection property, which is important for us to be able to use those as building blocks for more general configurations involving areas of triangles. Our proof combines radial and orthogonal projection estimates.

We will need the following lemma, due to \cite[Theorem 5.6]{Mattilabook2015}. For a finite compactly supported measure $\nu$ in $\R^n$, define 
\[
\mathcal{I}_t(\nu)
:=
\int_{\mathbb{R}^n} |\widehat{\nu}(\xi)|^2 |\xi|^{t-d}\,d\xi.
\]
Recall that the Sobolev dimension of a measure is defined as 
\[
\dim_S(\nu):=\sup\{t:\mathcal{I}_t(\nu)<\infty\}.
\]

\begin{lem}[\cite{Mattilabook2015}, Theorem 5.6]\label{thminMattilabook}
    For any measure $\mu$ on $\R^d$
    \[
    \dim_H\{\theta\in S^{d-1}:\dim_S((\rho_{\theta})_*\mu)<t\}\le \max\{0, d-1+t-\dim_S\mu\}.
    \]
\end{lem}

\begin{thm}[$L^2$ projection building block for triangle areas]\label{thm:trianglebuildingblockintheplane}
Let $\alpha>1$. Let $\mu_1,\mu_2,\mu_3$ be $\alpha$-Frostman
measures supported on pairwise separated compact sets
$E_1,E_2,E_3 \subset \mathbb{R}^2$, respectively. Then
\[
(\mu_1 \times \mu_2)
\left(
\left\{
(x_1,x_2)\in E_1\times E_2 :
\left(\rho_{\pi_{x_1}(x_2)^\perp}\right)_* \mu_3 \in L^2(\mathbb{R})
\right\}
\right)>0.
\]
In fact, the set above has full $\mu_1\times\mu_2$ measure in
$E_1\times E_2$.
\end{thm}

We note that Theorem \ref{thm:strengthenedSY}, the strengthened statement of Shmerkin Yavicoli, follows from the building block in the theorem above. Note that for all $(x_1,x_2)\in E_1\times E_2$, $\pi_{x_1}(x_2)$ points in the direction of the line connecting $x_1$ and $x_2$. If one has that 
$$\left(\rho_{\pi_{x_1}(x_2)^\perp}\right)_* (\mu_3) \in L^2(\mathbb{R}),$$
then the orthogonal projection of $E_3$ in the direction of the line perpendicular to the fixed edge $(x_1,x_2)$ has positive measure. 
Then as in \cite{SY25}, that is the same as positive measure worth of areas of triangles pinned.

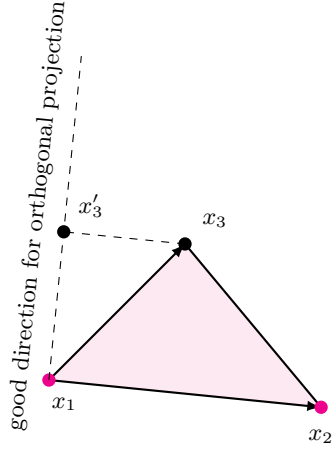
\begin{figure}[h]
\centering
\begin{tikzpicture}[scale=1.2,>=latex]
  \coordinate (x1) at (0,0);
  \coordinate (x2) at (3,-0.3);
  \coordinate (x3) at (1.5,1.5);

  \tikzset{
    vect/.style = {->, thick},
    pt/.style   = {circle, inner sep=0pt, minimum size=5pt}, 
    lbl/.style  = {font=\small}
  }

  \fill[magenta!10] (x1) -- (x2) -- (x3) -- cycle;

  \draw[vect] (x1) -- (x2);
  \draw[vect] (x1) -- (x3);
  \draw[thick] (x2) -- (x3);

  \node[pt,fill=magenta] at (x1) {};
  \node[pt,fill=magenta] at (x2) {};
  \node[pt,fill=black]    at (x3)   {};

  \node[lbl,below left=3pt]  at (0.5,0) {$x_1$};
  \node[lbl,below=5pt]       at (x2) {$x_2$};
  \node[lbl,above right=3pt] at (x3) {$x_3$};

  \coordinate (axisend) at ($(x1)!1.19!90:(x2)$);
  \draw[dashed] (x1) -- (axisend)
    node[midway,sloped,above=3pt,lbl]
    {good direction for orthogonal projection};

  \coordinate (proj) at ($(x1)!(x3)!(axisend)$);
  \draw[dashed,black] (x3) -- (proj);
  \node[pt,fill=black]             at (proj) {};
  \node[lbl,above right=2pt,black] at (proj) {$x_3'$};
\end{tikzpicture}
\caption{The geometry of good directions for orthogonal projections}
\end{figure}

\begin{proof}[Proof of Theorem \ref{thm:trianglebuildingblockintheplane}]
By the Orponen's radial projection estimate recalled in Theorem \ref{thm:OrponenL^pprojection}, there exists some $p=p(\alpha)>1$ such that
\[
\int \left\|(\pi_{x_1})_* \mu_2\right\|_{L^p(S^1)}^p \, d\mu_1(x_1)<\infty.
\]
In particular, for $\mu_1$-a.e. $x_1\in E_1$, the measure
$(\pi_{x_1})_*\mu_2$ is absolutely continuous with respect to
$\mathcal{H}^1|_{S^1}$. \hfill $(*)$

Let
\[
B_3
:=
\left\{
\theta\in S^1 :
\left(\rho_{\theta^\perp}\right)_*\mu_3 \notin L^2(\mathbb{R})
\right\}.
\]

Write
\[
\mu_{3,\theta}:=(\rho_\theta)_*\mu_3.
\]

By Plancherel, for a finite measure $\nu$ on $\mathbb{R}$,
\[
\nu\in L^2(\mathbb{R})
\quad\Longleftrightarrow\quad
\mathcal{I}_1(\nu)<\infty.
\]

Thus, if $\mu_{3,\theta}\notin L^2(\mathbb{R})$, then
$\dim_S(\mu_{3,\theta})\leq 1$.

Since the map $\theta\mapsto \theta^\perp$ is bi-Lipschitz on $S^1$,
we may use the exceptional set estimate for Sobolev dimensions of
orthogonal projections. By Lemma \ref{thminMattilabook}, for every sufficiently small $\epsilon>0$,
\[
\begin{aligned}
\dim_{\mathcal{H}}(B_3)
&\leq
\dim_{\mathcal{H}}
\left(
\left\{
\theta\in S^1 :
\dim_S(\mu_{3,\theta})<1+\epsilon
\right\}
\right)  \\
&\leq
1+(1+\epsilon)-\dim_S(\mu_3) \\
&=
2+\epsilon-\dim_S(\mu_3).
\end{aligned}
\]

Since $\mu_3$ is $\alpha$-Frostman, we have
$\dim_S(\mu_3)\geq \alpha$. Indeed, $\mathcal{I}_t(\mu_3)<\infty$
for every $t<\alpha$. 

Therefore
\[
\dim_{\mathcal H}(B_3)\leq 2+\epsilon-\alpha.
\]

Choose $\epsilon>0$ so small that $\epsilon<\alpha-1$. Then
\[
\dim_{\mathcal H}(B_3)<1,
\]
and hence
\[
\mathcal{H}^1(B_3)=0.
\]

By $(*)$, for $\mu_1$-a.e. $x_1\in E_1$,
\[
\left((\pi_{x_1})_*\mu_2\right)(B_3)=0.
\]
Equivalently,
\[
\mu_2\left(\pi_{x_1}^{-1}(B_3)\right)=0.
\]
Thus, for $\mu_1$-a.e. $x_1\in E_1$ and for $\mu_2$-a.e.
$x_2\in E_2$,
\[
\pi_{x_1}(x_2)\notin B_3.
\]
By the definition of $B_3$, this means that
\[
\left(\rho_{\pi_{x_1}(x_2)^\perp}\right)_*\mu_3
\in L^2(\mathbb{R})
\]
for $\mu_1\times\mu_2$-a.e. $(x_1,x_2)\in E_1\times E_2$.

Therefore
\[
(\mu_1 \times \mu_2)
\left(
\left\{
(x_1,x_2)\in E_1\times E_2 :
\left(\rho_{\pi_{x_1}(x_2)^\perp}\right)_* \mu_3 \in L^2(\mathbb{R})
\right\}
\right)
=
\mu_1(E_1)\mu_2(E_2)>0.
\]
\end{proof}

\subsection{Applications to triangle chains and fish graph}

By using the $L^2$ building block for triangle areas combined with the edge gluing mechanism proved by induction in Proposition  \ref{prop:trianglesbyedges} in the Appendix, we can get a result on edge-attached triangle chains stated in Theorem \ref{cor:edgetrianglechains}.

Indeed, notice that the hypothesis of Proposition \ref{prop:trianglesbyedges} is satisfied for $\alpha_0=1$ due to our $L^2$ projection building block in Theorem \ref{thm:trianglebuildingblockintheplane}. Then, for $E\subset \R^2$ with $\dim(E)>1$ take an $\alpha$-Frostman measure with $\alpha\in (1,\dim(E))$, and by standard separation Lemma (iterate \cite[Lemma 3]{BIO23} for example), one can find separated subsets $\{E_i\}_{1\leq i\leq k+2}$ with $\mu(E_i)>0$, and one can get an $\alpha$-Frostman measure $\mu_i$ in $E_i$ by restricting $\mu$ to $E_i$, that is, $\mu_i=\mu_{E_i}$, as in Definition \ref{def: restrictedmeasure}.

 We also note that the building block in Theorem \ref{thm:trianglebuildingblockintheplane}, implies in particular that
$$\mu_{1} \left(  \{x_1 \in E_1: \mathcal{L}^1\left( \A_{x_1}(E_2,E_3) \right) >0 \} \right) >0,$$
 which, combined with Proposition \ref{prop: trianglesbyvertices} in the Appendix, implies Theorem \ref{cor: vertextrianglechain}.

As our last application on the plane, we show how to obtain a positive measure of the areas of triangles realized in the fish graph.

\begin{proof}[Proof of Theorem \ref{thm: fishAllD} in the case $d=2$]

Given $E\subset \R^2$ with $\dim(E)>1$, find disjoint separated subsets $E_1,E_2,\dots ,E_6$ with $\mu(E_i)>0$ where $\mu$ is an $\alpha$-Frostman measure in $E$ with $\alpha>1$. Let us check that there exists $(x_1,x_2)\in E_1\times E_2$ (in fact, an abundance of them) such that
\[
\mathcal{L}^2 \left((A_{\Delta})^{fish}_{x_1,x_2}(E_3,E_4,E_5,E_6)\right)>0.
\]

 Intuitively, one wants to glue the fish tail made of vertices $\{v_4,v_5,v_6\}$ at the vertex $v_4$ in the edge-attached $2$-triangle chain induced by the vertices $v_1,v_2,v_3$ and $v_4$. To do that, one starts by pruning $E_4$ to a compact subset $E_4'$ such that $\mu_{E_4}(E_4')>0$ and such that   $(A_{\Delta})_{x_4}(E_5,E_6)$ has positive measure for all $x_4\in E_4'$. Then one uses Proposition \ref{prop:trianglesbyedges} for $E_1,E_2,E_3,E_4'$ obtaining 
   $$\mathcal{L}^2(\{(\A(x_1,x_2,x_3),\A(x_2,x_3,x_4))\colon x_3\in E_3, x_4\in E_4'\})>0,$$
for $(x_1,x_2)$ in a positive $\mu_{E_1}\times \mu_{E_2}$ measure subset $ G_{12}\subset E_1\times E_2$.

A simple Fubini argument finishes the proof.
\end{proof}

\subsection{Projection arguments in higher $d$}\label{subsection:projectionsinhigherd}

It is natural to ask what can be obtained in higher dimensions using
projection arguments. Using Marstrand's slicing theorem \cite[Theorem~6.7]{Mattilabook2015} we obtain a result in $\R^d$ under the threshold $d-1$. This corollary
is most relevant in the case $d=3$, since when $d \ge 4$ the threshold can be
further improved to $\frac{d+2}{2}$ using the Jacobian method developed in
Section~\ref{sec: jacobianmethod}; see Theorem~\ref{thm: k pinned fans} with $k=2$.

\begin{cor}

    Let $d\geq 2$, and let $E\subset \R^d$ be a compact set with $\dim(E)>d-1$, then there exists $x_1,x_2\in E$ such that 
$$\mathcal{L}^1((A_\Delta)_{x_1,x_2}(E))>0.$$
\end{cor}

\begin{proof}
We give a direct proof for \(d\geq 3\). The idea is to slice \(E\) by
parallel lines. This reduces the problem to finding a pinned distance set of
positive measure inside a hyperplane.

Choose \(s\) with
\[
d-1<s<\dim(E).
\]
By Marstrand's slicing theorem, for almost every direction
\(\theta\in S^{d-1}\), the set of base points \(a\in \theta^\perp\) for which
the line \(a+\mathbb R\theta\) has a nontrivial slice of \(E\), in the sense
that
\[
\dim\bigl(E\cap(a+\mathbb R\theta)\bigr)\geq s-(d-1)>0,
\]
has positive \((d-1)\)-dimensional measure in \(\theta^\perp\).

Fix such a direction \(\theta\), and denote this positive-measure set of base
points by \(F\subset \theta^\perp\). Since \(F\) has positive
\((d-1)\)-dimensional measure, there exists \(a_{12}\in F\) such that
\[
\mathcal L^1(\Delta_{a_{12}}(F))>0,
\qquad
\Delta_{a_{12}}(F):=\{|a_3-a_{12}|:a_3\in F\}.
\]
Because \(a_{12}\in F\), the slice \(E\cap(a_{12}+\mathbb R\theta)\) has
positive Hausdorff dimension, and hence contains two distinct points
\(x_1,x_2\).

Now take any \(a_3\in F\), and choose a point
\[
x_3\in E\cap(a_3+\mathbb R\theta).
\]
The points \(x_1,x_2\) lie on the line \(a_{12}+\mathbb R\theta\), while
\(x_3\) lies on the parallel line \(a_3+\mathbb R\theta\). Therefore the height
of \(x_3\) over the edge \((x_1,x_2)\) is exactly \(|a_3-a_{12}|\), and
\[
\mathrm{Area}(x_1,x_2,x_3)
=
\frac12 |x_1-x_2|\,|a_3-a_{12}|.
\]
Since the set of values \(|a_3-a_{12}|\), \(a_3\in F\), has positive
Lebesgue measure, we conclude that
\[
\mathcal L^1\bigl((A_\Delta)_{x_1,x_2}(E)\bigr)>0.
\]
\end{proof}
\begin{rem}

Alternatively, one can reduce the case $d\geq 3$ to $d=2$ by slicing by two planes. This is how it goes. By Mattila's slicing theorem, for any
\[
s \in (d-2,\dim(E)),
\]
and for \(\gamma_{d,d-2}\)-almost every \((d-2)\)-dimensional subspace
\(V \in G(d,d-2)\), one has
\[
\mathcal{H}^{d-2}\Bigl(
\Bigl\{ a \in V :
\dim\bigl(E \cap (V^\perp + a)\bigr) \geq s-(d-2)
\Bigr\}
\Bigr) > 0.
\]
Since \(\dim(E)>d-1\), we may choose \(s>d-1\). Then
\(s-(d-2)>1\), so there exists a two-dimensional affine plane
\(H\subset \mathbb R^d\) such that
\[
\dim(E\cap H)>1.
\]

Applying the planar case to the compact set \(E \cap H\subset H\simeq
\mathbb R^2\), we obtain points \(x_1,x_2\in E \cap H\) such that
\[
\mathcal L^1\bigl((A_\Delta)_{x_1,x_2}(E\cap H)\bigr)>0.
\]
Since
\[
(A_\Delta)_{x_1,x_2}(E\cap H)
\subset
(A_\Delta)_{x_1,x_2}(E),
\]
the desired conclusion follows.
\end{rem}

\appendix
\markboth{Appendix}{Appendix}

\section{Structural Theorems for Triangle Areas}\label{structuralThm}

In this appendix, we use chains as a model example to illustrate how Fubini-type arguments can leverage results on triangle areas, with one or two pins, as building blocks for deriving corresponding results on triangle areas realized by chains of triangles in both the vertex-attached and edge-attached settings. Our approach may be viewed as a triangle-chain analogue of the chain results of Ou and Taylor, obtained via pinned distance sets \cite{OT22}. For clarity and readability, we follow the slightly more streamlined presentation from \cite{OT22}, in which vertices are drawn from separated sets, while adapting the statements to the triangle-chain setting and highlighting the role played by triangle areas.

In the case of edge-attached chains of triangles, we only obtain results in the plane $(d=2)$ since we are using orthogonal projections onto lines to recover triangle areas. However, dimension $2$ is where such results will be the most useful for us, since in dimensions $d\geq 3$ the Jacobian method developed in Section \ref{sec: jacobianmethod} can induce results for chains of triangle areas sharing edges more directly under the same threshold that one has for $2$-stars.

To simplify the notation, we will denote $\mathcal{L}^k(A)$ alternatively by $|A|_{k}$.

\subsection{Single-pinned triangle chain with triangles attached at the vertices}

 Given points $x,y,z \in \R^d$, the Area function $\text{Area}(x,y,z)$ computes the area of the triangle formed by the points $x,y,z$. We define the $k-$chain of areas of triangles attached at the vertices as the map
$$A^{TC_k^{vtx}}(x_1,x_2,\hdots, x_{2k+1}) := \left( \A(x_1,x_2,x_3), \A(x_{3},x_{4},x_{5}), \hdots, \A(x_{2k-1},x_{2k},x_{2k+1}) \right) $$
 and the pinned $k-$chain of areas of triangles as 
$$ A^{TC_k^{vtx}}_{x}(x_2, \hdots, x_{2k+1} ):= A^{TC_k^{vtx}} (x,x_2, \hdots, x_{2k+1}) .$$

We will also denote 
$$A_x(E_2,E_3):=\{\A(x,x_2,x_3)\colon x_2\in E_2,x_3\in E_3\}$$
and 
$$ A^{TC_k^{vtx}}_x(E_2,\dots ,E_{2k+1}):=\{\A^{TC_k^{vtx}}_{x}(x_2, \hdots, x_{2k+1})\colon x_i\in E_i,\, 2\leq i\leq 2k+1\} .$$

\begin{prop}[Vertex attached triangle-chains from pinned triangle areas]
\label{prop: trianglesbyvertices}
Let $d \geq 2$, and $\alpha_0 >0$. Suppose that for any $\alpha>\alpha_0$ and any $\alpha$-Frostman measures $\mu_1,\mu_2, \mu_3$ supported in separated compact sets $E_1,E_2, E_3\subset \R^d$ respectively, one has abundance of good pins in $E_1$ for areas realized in $E_2,E_3$ in the sense that
$$  \mu_{1} \left(  \{x \in E_1: \left| A_{x}(E_2,E_3) \right|_1 >0 \} \right) >0.$$
Then, for any $k\geq 1$, any $\alpha>\alpha_0$, and any given $(2k+1)$ separated compact sets $E_1,E_2, \dots ,E_{2k+1}$ in $\R^d$ equipped with $\alpha$-dimensional Frostman measures $\mu_1,\mu_2,
\dots ,\mu_{2k+1}$, respectively, one has 
$$ \mu_1\left(\{x\in E_1\colon \left|  A^{TC_k^{vtx}}_x(E_2,\hdots, E_{2k+1})  \right|_k >0\}\right)>0.$$
\end{prop}

\begin{proof}
    The proof uses induction in $k$. The case $k=1$ is true by assumption. Now let $k\geq 2$ and assume the statement is true for $k-1$.

    By the induction hypothesis applied to the separated compact sets $E_{3},E_4,\dots ,E_{2k+1}$, we have that there exists a set
\[
G_3\subset E_3
\]
such that
\[
\mu_3(G_3)>0
\]
and for every \(z\in G_3\),
\[
\left|
A_z^{TC_{k-1}^{vtx}}
(E_4,\dots,E_{2k+1})
\right|_{k-1}>0.
\]

    Choose a compact subset \(K_3\subset G_3\) such that
$\mu_3(K_3)>0.$

    By assumption (case $k=1$) applied to the compact sets $E_1$, $E_2$ and $K_3$ equipped with measures $\mu_1,\mu_2,(\mu_3)_{K_3}$, there exists \(E_1^{(k)}\subset E_1\) such that
\[
\mu_1(E_1^{(k)})>0
\]
and for every \(x\in E_1^{(k)}\),
\[
\mathcal L^1\big(\A_x(E_2,K_3)\big)>0.
\]

    Consider a point $x \in E_1^{(k)}$. Notice that
    
    $$ \left|  A^{TC_k^{vtx}}_x(E_2,E_3,\dots ,E_{2k+1}) \right|_k \geq \int_{R} \left| \left[ A^{TC_k^{vtx}}_x(E_2,E_3, \dots, E_{2k+1}) \right]_a \right|_{k-1} \mathrm{d}a, $$
    where $R = \A_x(E_2,K_3 )$ (i.e., set of areas of triangles formed by the pinned point $x \in E_1$, and points in the sets $E_2$ and $K_3\subset E_{3} $), and 
$$ \left[ A^{TC_k^{vtx}}_x(E_2,E_3,\dots, E_{2k+1}) \right]_a := \lbrace \vec{v}\in \R^{k-1}: (a, \vec{v}) \in  A^{TC_k^{vtx}}_x(E_2,E_3, \dots ,E_{2k+1})  \rbrace \subset \R^{k-1}.$$

Since $|R|_1>0$, it is enough to show the slices satisfy $$\displaystyle{ \left| \left[ A^{TC_k^{vtx}}_x (E_2,E_3, \dots ,E_{2k+1}) \right]_a \right|_{k-1} >0} \text{ for all } a \in R.$$

Given $a\in R$, consider $x_a \in K_3\subset G_3$, and $y_a \in E_2$ such that $\text{Area}(x,y_a,x_a)=a$, then by definition of $G_3$, we have

$$ \left| \left[ A^{TC_k^{vtx}}_x (E_2,E_3, \dots ,E_{2k+1}) \right]_a \right|_{k-1} \geq \left| A^{TC_{k-1}^{vtx}}_{x_a}( E_{4},E_{5},\dots ,E_{2k+1}) \right|_{k-1} > 0, $$
  which finishes the proof of the proposition.  
\end{proof}

\subsection{Double-pinned triangle chain with triangles attached at edges}

Using similar notation as above, we can define an area set associated with $k$-triangle chains attached by edges as
$$A^{TC_k^{edge}}(x_1,x_2,\hdots, x_{k+2}) = \left( \A(x_1,x_2,x_3), \A(x_{2},x_{3},x_{4}), \hdots, \A(x_{k},x_{k+1},x_{k+2}) \right).$$ We can also define the double pinned $k-$chain of areas of triangles attached at the edges as the map 
$$ (A^{TC_k^{edge}})_{x,y} (x_3, \hdots, x_{k+2})= \A^{TC_k^{edge}}(x,y,x_3, \hdots, x_{k+2}).$$

\begin{prop} \label{prop:trianglesbyedges} 
 Suppose that $\alpha_0\in (0,2)$ has the following property: For any $\alpha>\alpha_0$ and $\alpha$-Frostman measures $\mu_1,\mu_2, \mu_3$ supported in separated compact sets $E_1,E_2, E_3\subset \R^2$ respectively, one has abundance of pairs of pins in $E_1\times E_2$ for which the orthogonal projection of $\mu_3$ belongs to $L^2$, that is,
$$  \mu_{1} \times \mu_{2} \left( \left\{ (x,y) \in E_1 \times E_2 : (\rho_{l_{xy}^{\perp}})_*(\mu_{3})\in L^2(\R)  \right\} \right) >0 $$
where $\rho_{l_{xy}^{\perp}}$ is the orthogonal projection on the direction perpendicular to the line connecting $x$ and $y$.

 Then, for any $k \geq 1$, any $\alpha>\alpha_0$ and any separated compact sets $E_1,E_2, \cdots, E_{k+2} \subset \R^2$ equipped with $\alpha$-Frostman measures $\mu_1,\mu_2,\dots ,\mu_{k+2}$, respectively, one has
 $$\mu_1\times \mu_2\left(\left\{(x,y)\in E_1\times E_2\colon \left| (A^{TC_k^{edge}})_{x,y}(E_3,\cdots, E_{k+2})  \right|_k >0\right\}\right)>0, $$
where $(A^{TC_k^{edge}})_{x,y}(E_3,\dots, E_{k+2})=\{A^{TC_k^{edge}}(x,y,x_3,\dots, x_{k+2})\colon x_i\in E_i,3\leq i\leq k+2\}$.
\end{prop}

\begin{proof}[Proof of Proposition \ref{prop:trianglesbyedges}]
The proof uses induction in $k\geq 1$, and the base case $k=1$ is true by assumption. Now assume that $k\geq 2$ and that the statement is true for $k-1$. Take $k+2$ separated compact sets $E_1,E_2,\dots, E_{k+2}\subset \R^2$, and $\alpha$-Frostman measures $\mu_i$ supported on $E_i$.

First apply the induction hypothesis to $E_2,E_3,E_4, \dots ,E_{k+2}$. Then 
$$\mu_2\times \mu_3\left(E_{23}:=\left\{(x_2,x_3)\in E_2\times E_3\colon \left|  (\A^{TC_{k-1}^{edge}})_{x_2,x_3}(E_4,\cdots, E_{k+2})  \right|_{k-1} >0\right\}\right)>0.$$ 

Let $$c:=\frac{(\mu_2\times \mu_{3})(E_{23})}{\mu_2(E_2)\mu_3(E_3)}>0$$

There exists a set $E_2'\subset E_2$, with $\mu_2(E_2')>0$ such that 
for all $x_2\in E_2'$, one has 
$$\mu_{3}(\{x_3\in E_3\colon (x_2,x_3)\in E_{23}\})\geq \frac{c}{2}\mu_3(E_3).$$
Define $E_3'(x_2):=\{x_3\in E_3\colon (x_2,x_3)\in E_{23}\}$. Then for all $x_2\in E_2',$ one has $\mu_3(E_3'(x_2))>0$ and

$$\left| (\A^{TC_{k-1}^{edge}})_{x_2,x_3}(E_4,\cdots, E_{k+2}) \right|_{k-1} >0,\,\text{for all }  x_3\in E_{3}'(x_2).$$

\definecolor{diskgray}{gray}{0.85}
\definecolor{innerblue}{RGB}{180,200,255}

\tikzset{
  ball/.style   = {circle, draw=none, fill=diskgray, minimum size=1.6cm},
  inner/.style  = {circle, draw=none, fill=innerblue, minimum size=1.1cm},
  vertex/.style = {circle, fill=blue, draw=none, inner sep=1pt}
}

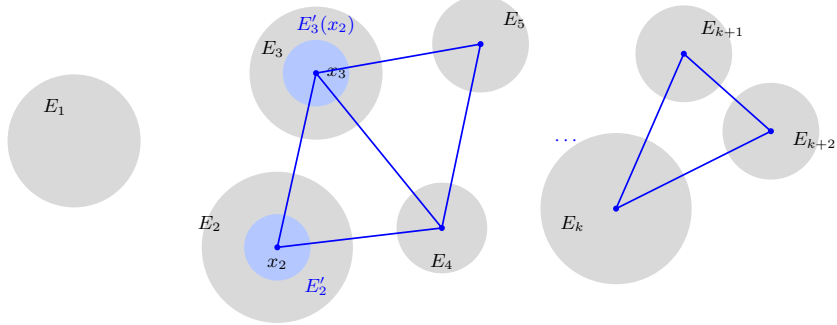
\begin{figure}[h]
\centering
\scalebox{0.8}{
\begin{tikzpicture}[scale=1.6, every node/.style={font=\small}]

  \coordinate (E1)      at (-3.0,  0.0);
  \coordinate (E2)      at (-0.9, -1.1);
  \coordinate (E3)      at (-0.5,  0.7);
  \coordinate (E4)      at ( 0.8, -0.9);
  \coordinate (E5)      at ( 1.2,  1.0);

  \coordinate (Ek)      at ( 2.6, -0.7);
  \coordinate (Ekplus1) at ( 3.3,  0.9);
  \coordinate (Ekplus2) at ( 4.2,  0.1);

  \node[ball,minimum size=2.2cm] at (E1)      {};
  \node[ball,minimum size=2.5cm]                    at (E2)      {};
  \node[ball,minimum size=2.2cm] at (E3)      {};
  \node[ball,minimum size=1.5cm] at (E4)      {};
  \node[ball]                    at (E5)      {};
  \node[ball,minimum size=2.5cm] at (Ek)      {};
  \node[ball]                    at (Ekplus1) {};
  \node[ball]                    at (Ekplus2) {};

  \node at ($(E1)      + (-0.2,0.35)$) {$E_1$};
  \node at ($(E2)      + (-0.7,0.25)$) {$E_2$};
  \node at ($(E3)      + (-0.45,0.25)$) {$E_3$};
  \node at ($(E4)      + ( 0.00,-0.35)$) {$E_4$};
  \node at ($(E5)      + ( 0.35,0.25)$) {$E_5$};
  \node at ($(Ek)      + (-0.45,-0.15)$) {$E_k$};
  \node at ($(Ekplus1) + ( 0.40,0.25)$) {$E_{k+1}$};
  \node at ($(Ekplus2) + ( 0.45,-0.10)$) {$E_{k+2}$};

  \node[inner] at (E2) {};
  \node[inner] at (E3) {};

  \node[blue] at ($(E2) + (0.4,-0.4)$) {$E_2'$};
  \node[blue] at ($(E3) + (0.1, 0.5)$) {$E_3'(x_2)$};

  \node[vertex,label=below:$x_2$] at (E2) {};
  \node[vertex,label=right:$x_3$] at (E3) {};
  \node[vertex]                 at (E4) {};
  \node[vertex]                 at (E5) {};
  \node[vertex]                 at (Ek) {};
  \node[vertex]                 at (Ekplus1) {};
  \node[vertex]                 at (Ekplus2) {};

  \draw[blue,thick] (E2) -- (E3);
  \draw[blue,thick] (E2) -- (E4);
  \draw[blue,thick] (E3) -- (E4);
  \draw[blue,thick] (E3) -- (E5);
  \draw[blue,thick] (E4) -- (E5);

  \draw[blue,thick] (Ek)      -- (Ekplus1);
  \draw[blue,thick] (Ekplus1) -- (Ekplus2);
  \draw[blue,thick] (Ekplus2) -- (Ek);

  \node[blue] at (2.1,0.0) {$\cdots$};

\end{tikzpicture}}
\caption{Schematic configuration of the sets $E_1,\dots,E_{k+2}$ 
with special points $x_2\in E_2'$ and $x_3\in E_3'(x_2)$.}
\end{figure}

Next use the hypothesis with $k=1$ for the sets $E_1,E_2',E_3$, and $\alpha$-Frostman measures, $\mu_1,\mu_2':=(\mu_2)_{E_2'},\mu_3$ so that 
$$ (\mu_{1} \times \mu_{2}') \left( G_{12} \right):= (\mu_{1} \times \mu_{2}') \left( \left\{ (x_1,x_2) \in E_1 \times E_2' : (\rho_{l_{x_1x_2}^{\perp}})_*(\mu_3) \in L^2(\R) \right\} \right) >0. $$

Note that since $\mu_2'$ is a restriction of $\mu_2$ that also implies $\mu_1\times \mu_2(G_{12})>0$.

Next we use a Fubini argument to get positive $k$ dimensional measure of $k$-areas. More precisely, we claim that for any $(x_1,x_2)\in G_{12}$,
$$\left|  (\A^{TC_k^{edge}})_{x_1,x_2}(E_3'(x_2),E_4,\cdots, E_{k+2})  \right|_{k} >0.$$
    
   \tikzset{
  redinner/.style = {circle, draw=none, fill=red!35, minimum size=1cm}
}
    
\begin{figure}[h]
\centering
\scalebox{0.8}{
\begin{tikzpicture}[scale=1.6, every node/.style={font=\small}]

  \coordinate (E1)      at (-3.0,  0.0);
  \coordinate (E2)      at (-0.9, -1.1);
  \coordinate (E3)      at (-0.5,  0.7);
  \coordinate (E4)      at ( 0.8, -0.9);
  \coordinate (E5)      at ( 1.2,  1.0);

  \coordinate (Ek)      at ( 2.6, -0.7);
  \coordinate (Ekplus1) at ( 3.3,  0.9);
  \coordinate (Ekplus2) at ( 4.2,  0.1);

  \node[ball,minimum size=2.2cm] at (E3)      {};
  \node[ball,minimum size=1.5cm] at (E4)      {};
  \node[ball]                    at (E5)      {};
  \node[ball,minimum size=2.5cm] at (Ek)      {};
  \node[ball]                    at (Ekplus1) {};
  \node[ball]                    at (Ekplus2) {};

  \node at ($(E3)      + (-0.8,0.25)$) {$E_3$};
  \node at ($(E4)      + ( 0.00,-0.35)$) {$E_4$};
  \node at ($(E5)      + ( 0.35,0.25)$) {$E_5$};
  \node at ($(Ek)      + (-0.45,-0.15)$) {$E_k$};
  \node at ($(Ekplus1) + ( 0.40,0.25)$) {$E_{k+1}$};
  \node at ($(Ekplus2) + ( 0.45,-0.10)$) {$E_{k+2}$};

  \node[inner] at (E3) {};

  \node[blue] at ($(E3) + (0.1, 0.5)$) {$E_3'(x_2)$};



  \coordinate (X1) at ($(E1) + (0.15,-0.15)$);
  \coordinate (X2) at (E2);
  \coordinate (X3) at (E3);

  \node[vertex,label=below:$x_1$] at (X1) {};
  \node[vertex,label=below:$x_2$] at (X2) {};
  \node[vertex,label=right:$x_3$] at (X3) {};
  \node[vertex]                 at (E4) {};
  \node[vertex]                 at (E5) {};
  \node[vertex]                 at (Ek) {};
  \node[vertex]                 at (Ekplus1) {};
  \node[vertex]                 at (Ekplus2) {};

  \draw[blue,thick] (E2) -- (E3);
  \draw[blue,thick] (E2) -- (E4);
  \draw[blue,thick] (E3) -- (E4);
  \draw[blue,thick] (E3) -- (E5);
  \draw[blue,thick] (E4) -- (E5);

  \draw[red,thick] (X1) -- (X2);
  \draw[red,thick] (X1) -- (X3);
  \draw[red,thick] (X2) -- (X3);

  \draw[blue,thick] (Ek)      -- (Ekplus1);
  \draw[blue,thick] (Ekplus1) -- (Ekplus2);
  \draw[blue,thick] (Ekplus2) -- (Ek);

  \node[blue] at (2.1,0.0) {$\cdots$};

\end{tikzpicture}}
\caption{Configuration with fixed $(x_1,x_2)\in G_{12}$.}
\end{figure}
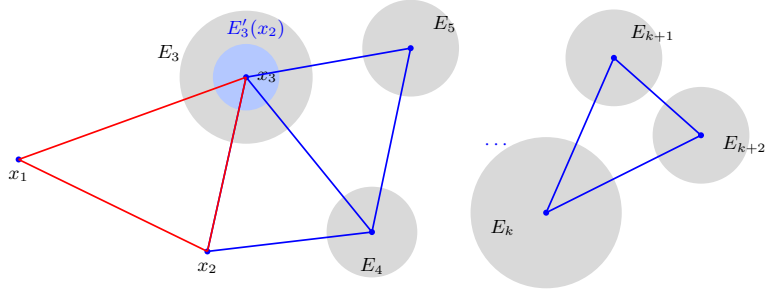

    Indeed,

    \begin{equation} \label{chain2}
     \begin{split}
       & \left| (\A^{TC_k^{edge}})_{x_1,x_2}(E_3'(x_2), E_4, \hdots, E_{k+2}) \right|_k \\
        & \geq\int_{R} \left| \left[ (\A^{TC_k^{edge}})_{x_1,x_2} (E_{3}'(x_2), E_{4}, \hdots, E_{k+2}) \right]_a \right|_{k-1} \mathrm{d} a,
        \end{split}
    \end{equation}
    where $R = \A_{x_1,x_2} \left ( {E_3'(x_2)} \right )$, and 

        $$ \left[ (\A^{TC_k^{edge}})_{x_1,x_2} (E_{3}'(x_2),E_4, \hdots, E_{k+2} ) \right]_a = \lbrace \vec{v}\in \R^{k-1}: (a, \vec{v}) \in  (\A^{TC_k^{edge}} )_{x_1,x_2}(E_{3}'(x_2),E_4 \hdots, E_{k+2} )  \rbrace $$
                    
    Note that for any $a\in R$, one can fix $x_{3}(a)\in E_3'(x_2)$ such that $\text{Area}(x_1,x_2,x_3(a))=a$, and then since $x_2\in E_2'$ and $x_{3}(a)\in E_3'(x_2)$, it holds that 

\begin{align*}
    0<&\left|  (\A^{TC_{k-1}^{edge}})_{x_2,x_3(a)}(E_4,\hdots, E_{k+2})  \right|_{k-1}\\
    \leq & \left| \left[ (\A^{TC_k^{edge}})_{x_1,x_2} (E_{3}'(x_2), E_{4}, \hdots, E_{k+2}) \right]_a \right|_{k-1}.
    \end{align*}

    Going back to (\ref{chain2}), we are reduced to checking that $R$ has positive Lebesgue measure. This follows because $\mu_3(E_3'(x_2))>0$ and $\mu_3|_{E_3'(x_2)}\leq \mu_3$, so the pushforward measure
$$(\rho_{l_{x_1x_2}^{\perp}})_{*}(\mu_3|_{E_3'(x_2)})$$
admits a nonzero $L^2$ density on $\R$. In particular, its support has positive Lebesgue measure. Since we are in the plane ($d=2$) that implies
$$|\A_{x_1,x_2}(E_3'(x_2))|_1>0.$$
\end{proof}

\printbibliography

@article{pinnedtrees,
title = {Nonempty interior of pinned distance and tree sets},
journal = {Advances in Mathematics},
volume = {493},
pages = {110917},
year = {2026},
issn = {0001-8708},
doi = {https://doi.org/10.1016/j.aim.2026.110917},
url = {https://www.sciencedirect.com/science/article/pii/S0001870826001398},
author = {Tainara Borges and Benjamin Foster and Yumeng Ou and Eyvindur Palsson}
}

@misc{BFOPRpaperI,
      title={Falconer-type results for any finite graph with multiple pins}, 
      author={Tainara Borges and Ben Foster and Yumeng Ou and Eyvindur Palsson and Francisco Romero Acosta},
      year={2026},
      eprint={2603.01954},
      archivePrefix={arXiv},
      primaryClass={math.CA},
      url={https://arxiv.org/abs/2603.01954}, 
}

@article {IPPS22,
    AUTHOR = {Iosevich, Alex and Pham, Minh-Quy and Pham, Thang and Shen,
              Chun-Yen},
     TITLE = {Pinned simplices and connections to product of sets on
              paraboloids},
   JOURNAL = {Indiana Univ. Math. J.},
  FJOURNAL = {Indiana University Mathematics Journal},
    VOLUME = {74},
      YEAR = {2025},
    NUMBER = {3},
     PAGES = {647--668},
      ISSN = {0022-2518,1943-5258},
   MRCLASS = {52C10 (28A80 42B10)},
  MRNUMBER = {4946877},
}

@article {Falconer85,
    AUTHOR = {Falconer, K. J.},
     TITLE = {On the {H}ausdorff dimensions of distance sets},
   JOURNAL = {Mathematika},
  FJOURNAL = {Mathematika. A Journal of Pure and Applied Mathematics},
    VOLUME = {32},
      YEAR = {1985},
    NUMBER = {2},
     PAGES = {206--212},
      ISSN = {0025-5793},
   MRCLASS = {28A75 (28A05)},
  MRNUMBER = {834490},
MRREVIEWER = {S.\ J.\ Taylor},
       DOI = {10.1112/S0025579300010998},
       URL = {https://doi.org/10.1112/S0025579300010998},
}

@misc{DORZ23,
      title={New improvement to Falconer distance set problem in higher dimensions}, 
      author={Xiumin Du and Yumeng Ou and Kevin Ren and Ruixiang Zhang},
      year={2024},
      eprint={2309.04103},
      archivePrefix={arXiv},
      primaryClass={math.CA},
      url={https://arxiv.org/abs/2309.04103}, 
}

@article {GIOW20,
    AUTHOR = {Guth, Larry and Iosevich, Alex and Ou, Yumeng and Wang, Hong},
     TITLE = {On {F}alconer's distance set problem in the plane},
   JOURNAL = {Invent. Math.},
  FJOURNAL = {Inventiones Mathematicae},
    VOLUME = {219},
      YEAR = {2020},
    NUMBER = {3},
     PAGES = {779--830},
      ISSN = {0020-9910,1432-1297},
   MRCLASS = {42B20 (28A80)},
  MRNUMBER = {4055179},
MRREVIEWER = {Jonathan\ MacDonald\ Fraser},
       DOI = {10.1007/s00222-019-00917-x},
       URL = {https://doi.org/10.1007/s00222-019-00917-x},
}

@article {Liu19,
    AUTHOR = {Liu, Bochen},
     TITLE = {An {$L^2$}-identity and pinned distance problem},
   JOURNAL = {Geom. Funct. Anal.},
  FJOURNAL = {Geometric and Functional Analysis},
    VOLUME = {29},
      YEAR = {2019},
    NUMBER = {1},
     PAGES = {283--294},
      ISSN = {1016-443X,1420-8970},
   MRCLASS = {28A75 (28A78 42B10)},
  MRNUMBER = {3925111},
MRREVIEWER = {Lars\ Olsen},
       DOI = {10.1007/s00039-019-00482-8},
       URL = {https://doi.org/10.1007/s00039-019-00482-8},
}

@book {Mattilabook2015,
    AUTHOR = {Mattila, Pertti},
     TITLE = {Fourier analysis and {H}ausdorff dimension},
    SERIES = {Cambridge Studies in Advanced Mathematics},
    VOLUME = {150},
 PUBLISHER = {Cambridge University Press, Cambridge},
      YEAR = {2015},
     PAGES = {xiv+440},
      ISBN = {978-1-107-10735-9},
   MRCLASS = {28-02 (28A15 28A78 28A80 42B10 60J65)},
  MRNUMBER = {3617376},
MRREVIEWER = {Benjamin\ Steinhurst},
       DOI = {10.1017/CBO9781316227619},
       URL = {https://doi.org/10.1017/CBO9781316227619},
}

@article {IMT12,
    AUTHOR = {Iosevich, Alex and Mourgoglou, Mihalis and Taylor, Krystal},
     TITLE = {On the {M}attila-{S}j\"olin theorem for distance sets},
   JOURNAL = {Ann. Acad. Sci. Fenn. Math.},
  FJOURNAL = {Annales Academi\ae\ Scientiarum Fennic\ae. Mathematica},
    VOLUME = {37},
      YEAR = {2012},
    NUMBER = {2},
     PAGES = {557--562},
      ISSN = {1239-629X,1798-2383},
   MRCLASS = {28A75 (42B20 52C10)},
  MRNUMBER = {2987085},
MRREVIEWER = {Alain\ Rivi\`ere},
       DOI = {10.5186/aasfm.2012.3732},
       URL = {https://doi.org/10.5186/aasfm.2012.3732},
}

@article {GILP15,
    AUTHOR = {Greenleaf, Allan and Iosevich, Alex and Liu, Bochen and
              Palsson, Eyvindur},
     TITLE = {A group-theoretic viewpoint on {E}rd\"os-{F}alconer problems
              and the {M}attila integral},
   JOURNAL = {Rev. Mat. Iberoam.},
  FJOURNAL = {Revista Matem\'atica Iberoamericana},
    VOLUME = {31},
      YEAR = {2015},
    NUMBER = {3},
     PAGES = {799--810},
      ISSN = {0213-2230,2235-0616},
   MRCLASS = {42B20 (52C10)},
  MRNUMBER = {3420476},
MRREVIEWER = {Tuomas\ P.\ Hyt\"onen},
       DOI = {10.4171/RMI/854},
       URL = {https://doi.org/10.4171/RMI/854},
}

@article {Orponen19,
    AUTHOR = {Orponen, Tuomas},
     TITLE = {On the dimension and smoothness of radial projections},
   JOURNAL = {Anal. PDE},
  FJOURNAL = {Analysis \& PDE},
    VOLUME = {12},
      YEAR = {2019},
    NUMBER = {5},
     PAGES = {1273--1294},
      ISSN = {2157-5045,1948-206X},
   MRCLASS = {28A80 (28A78 42C10)},
  MRNUMBER = {3892404},
       DOI = {10.2140/apde.2019.12.1273},
       URL = {https://doi.org/10.2140/apde.2019.12.1273},
}

@article {EIT11,
    AUTHOR = {Eswarathasan, Suresh and Iosevich, Alex and Taylor, Krystal},
     TITLE = {Fourier integral operators, fractal sets, and the regular
              value theorem},
   JOURNAL = {Adv. Math.},
  FJOURNAL = {Advances in Mathematics},
    VOLUME = {228},
      YEAR = {2011},
    NUMBER = {4},
     PAGES = {2385--2402},
      ISSN = {0001-8708,1090-2082},
   MRCLASS = {42B08 (28A80 44A12)},
  MRNUMBER = {2836125},
MRREVIEWER = {Herv\'e\ Queff\'elec},
       DOI = {10.1016/j.aim.2011.07.012},
       URL = {https://doi.org/10.1016/j.aim.2011.07.012},
}

@misc{BMS24,
title={Pinned Dot Product Set Estimates}, 
      author={Paige Bright and Caleb Marshall and Steven Senger},
      year={2024},
      eprint={2412.17985},
      archivePrefix={arXiv},
      primaryClass={math.CA},
      url={https://arxiv.org/abs/2412.17985},
}

@article {GIT21,
    AUTHOR = {Greenleaf, Allan and Iosevich, Alex and Taylor, Krystal},
     TITLE = {Configuration sets with nonempty interior},
   JOURNAL = {J. Geom. Anal.},
  FJOURNAL = {Journal of Geometric Analysis},
    VOLUME = {31},
      YEAR = {2021},
    NUMBER = {7},
     PAGES = {6662--6680},
      ISSN = {1050-6926,1559-002X},
   MRCLASS = {28A80 (35S30 44A12)},
  MRNUMBER = {4289240},
       DOI = {10.1007/s12220-019-00288-y},
       URL = {https://doi.org/10.1007/s12220-019-00288-y},
}

@article {GI12,
    AUTHOR = {Greenleaf, Allan and Iosevich, Alex},
     TITLE = {On triangles determined by subsets of the {E}uclidean plane,
              the associated bilinear operators and applications to discrete
              geometry},
   JOURNAL = {Anal. PDE},
  FJOURNAL = {Analysis \& PDE},
    VOLUME = {5},
      YEAR = {2012},
    NUMBER = {2},
     PAGES = {397--409},
      ISSN = {2157-5045,1948-206X},
   MRCLASS = {42B15 (52C10)},
  MRNUMBER = {2970712},
MRREVIEWER = {Andreas\ Seeger},
       DOI = {10.2140/apde.2012.5.397},
       URL = {https://doi.org/10.2140/apde.2012.5.397},
}

@incollection {EHI13,
    AUTHOR = {Erdogan, Burak and Hart, Derrick and Iosevich, Alex},
     TITLE = {Multiparameter projection theorems with applications to
              sums-products and finite point configurations in the
              {E}uclidean setting},
 BOOKTITLE = {Recent advances in harmonic analysis and applications},
    SERIES = {Springer Proc. Math. Stat.},
    VOLUME = {25},
     PAGES = {93--103},
 PUBLISHER = {Springer, New York},
      YEAR = {2013},
      ISBN = {978-1-4614-4565-4; 978-1-4614-4564-7},
   MRCLASS = {28A80 (42B08)},
  MRNUMBER = {3066881},
MRREVIEWER = {Li-Feng\ Xi},
       DOI = {10.1007/978-1-4614-4565-4\_11},
       URL = {https://doi.org/10.1007/978-1-4614-4565-4_11},
}

@article {PRA23,
    AUTHOR = {Palsson, Eyvindur Ari and Romero Acosta, Francisco},
     TITLE = {A {M}attila-{S}j\"olin theorem for triangles},
   JOURNAL = {J. Funct. Anal.},
  FJOURNAL = {Journal of Functional Analysis},
    VOLUME = {284},
      YEAR = {2023},
    NUMBER = {6},
     PAGES = {Paper No. 109814, 20},
      ISSN = {0022-1236,1096-0783},
   MRCLASS = {28A78 (42B20 52C10)},
  MRNUMBER = {4530888},
MRREVIEWER = {Vladimir\ Eiderman},
       DOI = {10.1016/j.jfa.2022.109814},
       URL = {https://doi.org/10.1016/j.jfa.2022.109814},
}

@article {PRA25,
    AUTHOR = {Palsson, Eyvindur Ari and Romero Acosta, Francisco},
     TITLE = {A {M}attila-{S}j\"olin theorem for simplices in low
              dimensions},
   JOURNAL = {Math. Ann.},
  FJOURNAL = {Mathematische Annalen},
    VOLUME = {391},
      YEAR = {2025},
    NUMBER = {1},
     PAGES = {1123--1146},
      ISSN = {0025-5831,1432-1807},
   MRCLASS = {28A75 (28A78 42B20 52A20)},
  MRNUMBER = {4846807},
MRREVIEWER = {Bochen\ Liu},
       DOI = {10.1007/s00208-024-02948-z},
       URL = {https://doi.org/10.1007/s00208-024-02948-z},
}

@article {GGIP15,
    AUTHOR = {Grafakos, Loukas and Greenleaf, Allan and Iosevich, Alex and
              Palsson, Eyvindur},
     TITLE = {Multilinear generalized {R}adon transforms and point
              configurations},
   JOURNAL = {Forum Math.},
  FJOURNAL = {Forum Mathematicum},
    VOLUME = {27},
      YEAR = {2015},
    NUMBER = {4},
     PAGES = {2323--2360},
      ISSN = {0933-7741,1435-5337},
   MRCLASS = {42B15 (05D05)},
  MRNUMBER = {3365800},
       DOI = {10.1515/forum-2013-0128},
       URL = {https://doi.org/10.1515/forum-2013-0128},
}

@article {GIT22,
    AUTHOR = {Greenleaf, Allan and Iosevich, Alex and Taylor, Krystal},
     TITLE = {On {$k$}-point configuration sets with nonempty interior},
   JOURNAL = {Mathematika},
  FJOURNAL = {Mathematika. A Journal of Pure and Applied Mathematics},
    VOLUME = {68},
      YEAR = {2022},
    NUMBER = {1},
     PAGES = {163--190},
      ISSN = {0025-5793,2041-7942},
   MRCLASS = {28A75 (28A80 52C10 58J40)},
  MRNUMBER = {4405974},
MRREVIEWER = {Xiumin\ Du},
       DOI = {10.1112/mtk.12114},
       URL = {https://doi.org/10.1112/mtk.12114},
}

@article {GIT24,
    AUTHOR = {Greenleaf, Allan and Iosevich, Alex and Taylor, Krystal},
     TITLE = {Nonempty interior of configuration sets via microlocal
              partition optimization},
   JOURNAL = {Math. Z.},
  FJOURNAL = {Mathematische Zeitschrift},
    VOLUME = {306},
      YEAR = {2024},
    NUMBER = {4},
     PAGES = {Paper No. 66, 20},
      ISSN = {0025-5874,1432-1823},
   MRCLASS = {28A75 (28A80 52C10 58J40)},
  MRNUMBER = {4716767},
MRREVIEWER = {Stefan\ Steinerberger},
       DOI = {10.1007/s00209-024-03466-z},
       URL = {https://doi.org/10.1007/s00209-024-03466-z},
}

@article {GIT25,
    AUTHOR = {Greenleaf, Allan and Iosevich, Alex and Taylor, Krystal},
     TITLE = {Realizing trees of configurations in thin sets},
   JOURNAL = {Pacific J. Math.},
  FJOURNAL = {Pacific Journal of Mathematics},
    VOLUME = {335},
      YEAR = {2025},
    NUMBER = {2},
     PAGES = {355--372},
      ISSN = {0030-8730,1945-5844},
   MRCLASS = {28A75 (42B35)},
  MRNUMBER = {4904870},
       DOI = {10.2140/pjm.2025.335.355},
       URL = {https://doi.org/10.2140/pjm.2025.335.355},
}

@article {GIP17,
    AUTHOR = {Greenleaf, Allan and Iosevich, Alex and Pramanik, Malabika},
     TITLE = {On necklaces inside thin subsets of {$\mathbb{R}^d$}},
   JOURNAL = {Math. Res. Lett.},
  FJOURNAL = {Mathematical Research Letters},
    VOLUME = {24},
      YEAR = {2017},
    NUMBER = {2},
     PAGES = {347--362},
      ISSN = {1073-2780,1945-001X},
   MRCLASS = {28A80 (43A46)},
  MRNUMBER = {3685274},
       DOI = {10.4310/MRL.2017.v24.n2.a4},
       URL = {https://doi.org/10.4310/MRL.2017.v24.n2.a4},
}

@article {McDonald21,
    AUTHOR = {McDonald, Alex},
     TITLE = {Areas spanned by point configurations in the plane},
   JOURNAL = {Proc. Amer. Math. Soc.},
  FJOURNAL = {Proceedings of the American Mathematical Society},
    VOLUME = {149},
      YEAR = {2021},
    NUMBER = {5},
     PAGES = {2035--2049},
      ISSN = {0002-9939,1088-6826},
   MRCLASS = {28A75},
  MRNUMBER = {4232196},
       DOI = {10.1090/proc/15348},
       URL = {https://doi.org/10.1090/proc/15348},
}

@article {GM22,
    AUTHOR = {Galo, Belmiro and McDonald, Alex},
     TITLE = {Volumes spanned by {$k$}-point configurations in {$\Bbb R^d$}},
   JOURNAL = {J. Geom. Anal.},
  FJOURNAL = {Journal of Geometric Analysis},
    VOLUME = {32},
      YEAR = {2022},
    NUMBER = {1},
     PAGES = {Paper No. 23, 26},
      ISSN = {1050-6926,1559-002X},
   MRCLASS = {28A75},
  MRNUMBER = {4349927},
MRREVIEWER = {Stefan\ Steinerberger},
       DOI = {10.1007/s12220-021-00763-5},
       URL = {https://doi.org/10.1007/s12220-021-00763-5},
}

@article {SY25,
    AUTHOR = {Shmerkin, Pablo and Yavicoli, Alexia},
     TITLE = {On the volumes of simplices determined by a subset of {$\Bbb
              R^d$}},
   JOURNAL = {Ann. Fenn. Math.},
  FJOURNAL = {Annales Fennici Mathematici},
    VOLUME = {50},
      YEAR = {2025},
    NUMBER = {1},
     PAGES = {97--108},
      ISSN = {2737-0690,2737-114X},
   MRCLASS = {28A78 (11B25 28A12 28A80)},
  MRNUMBER = {4877951},
       DOI = {10.54330/afm.159807},
       URL = {https://doi.org/10.54330/afm.159807},
}

@article {BIT16,
    AUTHOR = {Bennett, Michael and Iosevich, Alexander and Taylor, Krystal},
     TITLE = {Finite chains inside thin subsets of {$\Bbb{R}^d$}},
   JOURNAL = {Anal. PDE},
  FJOURNAL = {Analysis \& PDE},
    VOLUME = {9},
      YEAR = {2016},
    NUMBER = {3},
     PAGES = {597--614},
      ISSN = {2157-5045,1948-206X},
   MRCLASS = {28A75 (42B10)},
  MRNUMBER = {3518531},
MRREVIEWER = {Gareth\ Speight},
       DOI = {10.2140/apde.2016.9.597},
       URL = {https://doi.org/10.2140/apde.2016.9.597},
}

@incollection {IT19,
    AUTHOR = {Iosevich, A. and Taylor, K.},
     TITLE = {Finite trees inside thin subsets of {$\Bbb R^d$}},
 BOOKTITLE = {Modern methods in operator theory and harmonic analysis},
    SERIES = {Springer Proc. Math. Stat.},
    VOLUME = {291},
     PAGES = {51--56},
 PUBLISHER = {Springer, Cham},
      YEAR = {2019},
      ISBN = {978-3-030-26748-3; 978-3-030-26747-6},
   MRCLASS = {42B10 (05C05 28A80)},
  MRNUMBER = {4008977},
MRREVIEWER = {Peter\ R.\ Massopust},
       DOI = {10.1007/978-3-030-26748-3\_3},
       URL = {https://doi.org/10.1007/978-3-030-26748-3_3},
}

@article {OT22,
    AUTHOR = {Ou, Yumeng and Taylor, Krystal},
     TITLE = {Finite point configurations and the regular value theorem in a
              fractal setting},
   JOURNAL = {Indiana Univ. Math. J.},
  FJOURNAL = {Indiana University Mathematics Journal},
    VOLUME = {71},
      YEAR = {2022},
    NUMBER = {4},
     PAGES = {1707--1761},
      ISSN = {0022-2518,1943-5258},
   MRCLASS = {42B10 (28A78 52C10)},
  MRNUMBER = {4481098},
MRREVIEWER = {Rami\ Ayoush},
}

@misc{IMMM25,
      title={The VC-dimension and point configurations in $\mathbb{R}^d$}, 
      author={Alex Iosevich and Akos Magyar and Alex McDonald and Brian McDonald},
      year={2025},
      eprint={2510.13984},
      archivePrefix={arXiv},
      primaryClass={math.CA},
      url={https://arxiv.org/abs/2510.13984}, 
}

@misc{BIO23,
      title={A singular variant of the Falconer distance problem}, 
      author={Tainara Borges and Alex Iosevich and Yumeng Ou},
      year={2023},
      eprint={2306.05247},
      archivePrefix={arXiv},
      primaryClass={math.CA},
      url={https://arxiv.org/abs/2306.05247}, 
}

@article {LNPR,
    AUTHOR = {Lew, Alan and Nevo, Eran and Peled, Yuval and Raz, Orit E.},
     TITLE = {On the {$k$}-volume rigidity of a simplicial complex in {$\Bbb
              R^d$}},
   JOURNAL = {Forum Math. Sigma},
  FJOURNAL = {Forum of Mathematics. Sigma},
    VOLUME = {13},
      YEAR = {2025},
     PAGES = {Paper No. e195, 10},
      ISSN = {2050-5094},
   MRCLASS = {52C25},
  MRNUMBER = {4997997},
       DOI = {10.1017/fms.2025.10140},
       URL = {https://doi.org/10.1017/fms.2025.10140},
}

@article {CMSource,
    AUTHOR = {Liberti, Leo and Lavor, Carlile and Maculan, Nelson and
              Mucherino, Antonio},
     TITLE = {Euclidean distance geometry and applications},
   JOURNAL = {SIAM Rev.},
  FJOURNAL = {SIAM Review},
    VOLUME = {56},
      YEAR = {2014},
    NUMBER = {1},
     PAGES = {3--69},
      ISSN = {1095-7200,0036-1445},
   MRCLASS = {51K05 (51F15)},
  MRNUMBER = {3246296},
       DOI = {10.1137/120875909},
       URL = {https://doi.org/10.1137/120875909},
}

@article {HallSource,
    AUTHOR = {Cameron, Peter J.},
     TITLE = {Hall's marriage theorem},
   JOURNAL = {J. Lond. Math. Soc. (2)},
  FJOURNAL = {Journal of the London Mathematical Society. Second Series},
    VOLUME = {113},
      YEAR = {2026},
    NUMBER = {1},
     PAGES = {Paper No. e70378, 9},
      ISSN = {0024-6107,1469-7750},
   MRCLASS = {05D15 (05A16 05B15 05C70)},
  MRNUMBER = {5011584},
       DOI = {10.1112/jlms.70378},
       URL = {https://doi.org/10.1112/jlms.70378},
}

@article {KeletiShmerkin,
    AUTHOR = {Keleti, Tam\'as and Shmerkin, Pablo},
     TITLE = {New bounds on the dimensions of planar distance sets},
   JOURNAL = {Geom. Funct. Anal.},
  FJOURNAL = {Geometric and Functional Analysis},
    VOLUME = {29},
      YEAR = {2019},
    NUMBER = {6},
     PAGES = {1886--1948},
      ISSN = {1016-443X,1420-8970},
   MRCLASS = {28A75 (26A15 28A80 49Q15)},
  MRNUMBER = {4034924},
MRREVIEWER = {Lars\ Olsen},
       DOI = {10.1007/s00039-019-00500-9},
       URL = {https://doi.org/10.1007/s00039-019-00500-9},
}

@misc{GMP26,
      title={On volumes of simplices in intermediate dimensions}, 
      author={Jos\'{e} {Gaitan Montejo} and Eyvindur Palsson},
      year={2026},
      eprint={2605.22450},
      archivePrefix={arXiv},
      primaryClass={math.CA},
      url={https://arxiv.org/abs/2605.22450}, 
}

@unpublished{BOP2026,
  title={From weighted paraboloid restriction to $k$-stars and distance graphs},
  author={Borges, Tainara and Ou, Yumeng and Pasquariello, Marcus},
  year={2026},
  note={Forthcoming}
}

@article {DZ2019,
    AUTHOR = {Du, Xiumin and Zhang, Ruixiang},
     TITLE = {Sharp {$L^2$} estimates of the {S}chr\"odinger maximal
              function in higher dimensions},
   JOURNAL = {Ann. of Math. (2)},
  FJOURNAL = {Annals of Mathematics. Second Series},
    VOLUME = {189},
      YEAR = {2019},
    NUMBER = {3},
     PAGES = {837--861},
      ISSN = {0003-486X,1939-8980},
   MRCLASS = {42B20 (42B37)},
  MRNUMBER = {3961084},
MRREVIEWER = {Dong\ Dong},
       DOI = {10.4007/annals.2019.189.3.4},
       URL = {https://doi.org/10.4007/annals.2019.189.3.4},
}


\end{document}